\documentclass{article}

\usepackage[utf8]{inputenc}
\usepackage[T1]{fontenc}

\usepackage{lmodern}

\usepackage{amsmath}        
\usepackage{amsfonts}       
\usepackage{amssymb}
\usepackage{amsthm}         
\usepackage{bbding}         
\usepackage{bm}             
\usepackage{graphicx}       
\usepackage{fancyvrb}       
\usepackage{dcolumn}        
\usepackage{booktabs}       
\usepackage{paralist}       
\usepackage[usenames]{xcolor}  
\usepackage[numbers,sort&compress]{natbib}
\usepackage{comment}

\theoremstyle{plain}
\newtheorem{thm}{Theorem}[section]
\newtheorem{lemma}[thm]{Lemma}
\newtheorem{obs}[thm]{Observation}
\newtheorem{fct}[thm]{Fact}
\newtheorem{cor}[thm]{Corollary}
\newtheorem{prop}[thm]{Proposition}

\theoremstyle{definition}

\theoremstyle{remark}

\newtheorem{problem}[thm]{Problem}

\newcommand{\bin}{\{0,1\}}

\newcommand{\Mat}{\bin^{m\times n}}
\newcommand{\Pat}{\bin^{k\times \ell}}

\newcommand{\smm}[1]{\left({\begin{smallmatrix}#1\end{smallmatrix}}
\right)}

\newcommand{\im}{\preccurlyeq}

\newcommand{\Avm}[1]{Av_{\im}\left(#1\right)}
\newcommand{\Avmax}[1]{Av_{crit}\left(#1\right)}

\newcommand{\class}[1]{\mathcal{#1}}
\newcommand{\cF}{\class{F}}
\newcommand{\cC}{\class{C}}

\newcommand{\ovD}{\overline{D}} 

\newcommand{\rc}[1]{r\left(#1\right)}
\newcommand{\cc}[1]{c\left(#1\right)}
\newcommand{\rce}[2]{r\left(#1,#2\right)}

\newcommand{\rbpat}{\class{Q}}

\newcommand{\ex}{\mathrm{ex}}
\newcommand{\exm}{\mathrm{ex}_\im}

\DeclareMathOperator{\supp}{supp}

\begin{document}
\title{On the structure of matrices avoiding interval-minor 
patterns\thanks{Supported 
by project 
16-01602Y of the Czech Science Foundation (GA\v{C}R), and by the Neuron 
Foundation for the Support of Science.}}

\author{V\'it Jel\'inek
\thanks{Computer Science Institute, Charles University, Faculty of Mathematics 
and Physics, Malostransk\'e n\'am\v est\'i 25, Prague, Czech Republic, 
\texttt{jelinek@iuuk.mff.cuni.cz}}
\and
Stanislav Ku{\v c}era
\thanks{Department of Mathematics, London School of Economics, 
 Houghton Street, 
 London, WC2A 2AE, 
 United Kingdom
\texttt{s.kucera@lse.ac.uk}}
}
\date{}
\maketitle
\begin{abstract}
We study the structure of 01-matrices avoiding a pattern $P$ as an interval 
minor. We focus on critical $P$-avoiders, i.e., on the $P$-avoiding matrices in 
which changing a 0-entry to a 1-entry always creates a copy of $P$ as an 
interval minor. 

Let $Q$ be the $3\times 3$ permutation matrix corresponding to the 
permutation~$231$. As our main result, we show that for every pattern $P$ that 
has no rotated copy of $Q$ as interval minor, there is a constant $c_P$ such 
that any row and any column in any critical $P$-avoiding matrix can be 
partitioned into at most $c_P$ intervals, each consisting entirely of 0-entries 
or entirely of 1-entries. In contrast, for any pattern $P$ that contains a 
rotated copy of~$Q$, we construct critical $P$-avoiding matrices of arbitrary 
size $n\times n$ having a row with $\Omega(n)$ alternating intervals of 
0-entries and 1-entries.
\end{abstract}

\newsavebox{\smlmat}
\savebox{\smlmat}{$\smm{\bullet&\bullet\\\bullet& }$}
\newsavebox{\smlmatb}
\savebox{\smlmatb}{$\smm{\bullet&\bullet\\\bullet&\bullet}$}
\newsavebox{\smlmatc}
\savebox{\smlmatc}{$\smm{\bullet&\bullet&\bullet\\ &\bullet& }$}

\section{Introduction}
A \emph{binary matrix} is a matrix with entries equal to 0 or 1. All matrices 
considered in this paper are binary.
The study of extremal problems of binary matrices has been initiated by the 
papers of Bienstock and Gy\H ori~\cite{bienstock} and of Füredi~\cite{furedi}. 
Since these early works, most of the research in this area has focused on the 
concept of forbidden submatrices: a matrix $M$ is said to contain a pattern $P$ 
as a submatrix if we can transform $M$ into $P$ by deleting some rows and 
columns, and by changing 1-entries into 0-entries. This notion of submatrix is a 
matrix analogue of the notion of subgraph in graph theory. 


The main problem in the study of pattern-avoiding matrices is to determine the 
extremal function $\ex(n;P)$, defined as the largest number of 1-entries in an 
$n\times n$ binary matrix avoiding the pattern $P$ as submatrix. This is an 
analogue of the classical Turán-type problem of finding a largest number of 
edges in an $n$-vertex graph avoiding a given subgraph. Despite the analogy, the 
function $\ex(n;P)$ may exhibit an asymptotic behaviour not encountered in Turán 
theory. For instance, for the pattern\footnote{We use the convention of 
representing 1-entries in binary matrices by dots and 0-entries by 
blanks.} $P=\smm{\bullet&&\bullet&\\&\bullet&&\bullet}$ Füredi and 
Hajnal~\cite{furedi_hajnal} proved that $\ex(n;P)=\Theta(n\alpha(n))$, where 
$\alpha(n)$ is the inverse of the Ackermann function.

The asymptotic behaviour of $\ex(n;P)$ for general $P$ is still not well 
understood. Füredi and Hajnal~\cite{furedi_hajnal} posed the problem of 
characterising the \emph{linear} patterns, i.e., the patterns $P$ satisfying 
$\ex(n;P)=O(n)$. Marcus and Tardos~\cite{marcus} proved that $\ex(n;P)=O(n)$ 
whenever $P$ is a \emph{permutation matrix}, i.e., $P$ has exactly one 1-entry 
in each row and each column. This result, combined with previous work of 
Klazar~\cite{klazar_fh}, has confirmed the long-standing Stanley--Wilf 
conjecture. However, the problem of characterising linear patterns is
still open despite a number of further partial
results~\cite{crowdmath,fulek,keszegh,tardos,geneson,pettie_degrees}.

Fox~\cite{fox} has introduced a different notion of containment among binary 
matrices, based on the concept of interval minors. Informally, a matrix $M$ 
contains a pattern $P$ as an interval minor if we can transform $M$ into $P$ by 
contracting adjacent rows or columns and changing 1-entries into 0-entries; see 
Section~\ref{sec-prelims} for the precise definition. In this paper, we mostly 
deal with containment and avoidance of interval minors rather than submatrices. 
Therefore, the phrases \emph{$M$ avoids $P$} or \emph{$M$ contains $P$} always 
refer to avoidance or containment of interval minors, and the term 
\emph{$P$-avoider} always refers to a matrix that 
avoids $P$ as interval minor.

In analogy with $\ex(n;P)$, it is natural to consider the corresponding 
extremal function $\exm(n;P)$ as the largest number of 1-entries in an $n\times 
n$ matrix that avoids $P$ as an interval minor. If $M$ contains $P$ as a 
submatrix, it also contains it as an interval minor, and therefore 
$\exm(n;P)\le\ex(n;P)$. Moreover, it can be easily seen that for a permutation 
matrix $P$ the two notions of containment are equivalent, and hence 
$\exm(n;P)=\ex(n;P)$.

Fox~\cite{fox} used interval minors as a key tool in his construction of 
permutation patterns with exponential Stanley--Wilf limits. In view of the 
results of Cibulka~\cite{cibulka09}, this is equivalent to constructing a 
permutation matrix $P$ for which the limit of the ratio $\ex(n;P)/n$ (which is 
equal to $\exm(n;P)/n$) is exponential in the size of~$P$. 

Even before the work of Fox, interval minors have been implicitly used by 
Guillemot and Marx~\cite{gm}, who proved that a permutation matrix $M$ which 
avoids as interval minor a fixed complete square pattern (i.e., a square 
pattern with all entries equal to 1) admits a type of recursive decomposition 
of bounded complexity. This result can be viewed as an analogue of the grid 
theorem from graph theory~\cite{rs}, which states that graphs avoiding a large 
square grid as a minor have bounded tree-width. Guillemot and Marx used their 
result on forbidden interval minors to design a linear-time algorithm for 
testing the containment of a fixed pattern in a permutation.

Subsequent research into interval-minor avoidance has focused on avoiders of a 
complete matrix. In particular, Mohar et al.~\cite{mohar} obtained exact values 
for the extremal function for matrices simultaneously avoiding a complete 
pattern of size $2\times \ell$ and its transpose, and they obtained bounds for 
patterns of size $3\times \ell$. Their results were further generalised by Mao 
et al.~\cite{mao} to a multidimensional setting.

While the functions $\ex(n;P)$ exhibit diverse forms of asymptotic behaviour, 
the function $\exm(n;P)$ is linear for every nontrivial pattern~$P$. This is 
a consequence of the Marcus--Tardos theorem and the fact that any binary matrix 
is an interval minor of a permutation matrix; see Fox~\cite{fox}. Therefore, 
in the interval-minor avoidance setting, it is not as natural to classify 
patterns by the growth of~$\exm(n;P)$ alone as in the submatrix avoidance 
setting. 

In our paper, we instead classify the patterns $P$ based on the structure of 
the $P$-avoiders. We introduce the notion of \emph{line complexity} of a binary 
matrix $M$, as the largest number of maximal runs of consecutive 0-entries in a 
single row or a single column of~$M$. We focus on the \emph{critical} 
$P$-avoiders, which are the matrices that avoid $P$ as interval minor, but lose 
this property when any 0-entry is changed into a 1-entry. 

Our main result is a sharp dichotomy for line complexity of critical 
$P$-avoiders. Let $Q_1,\dotsc,Q_4$ be defined as follows:
\[
Q_1=\smm{ &\bullet& \\\bullet& & \\ & &\bullet},\ Q_2=\smm{ &\bullet& \\ & 
&\bullet\\\bullet& & },\ Q_3=\smm{\bullet& & \\ & &\bullet\\ &\bullet& 
}\text{ and } 
Q_4=\smm{ & &\bullet\\\bullet& & \\ &\bullet& }.
\]
We show that if a pattern $P$ avoids the four patterns $Q_i$ as interval minors 
(or equivalently, as submatrices), then the line-complexity of every critical 
$P$-avoider is bounded by a constant $c_P$ depending only on~$P$. On the other 
hand, if $P$ contains at least one of the $Q_i$, then there are critical 
$P$-avoiders of size $n\times n$ with line complexity $\Omega(n)$, for any~$n$.

After properly introducing our terminology and proving several simple basic 
facts in Section~\ref{sec-prelims}, we devote Section~\ref{sec-bounded} to 
the statement and proof of our main result. In Section~\ref{sec-further}, we 
discuss the possibility of extending our approach to general minor-closed matrix 
classes, and present several open problems.

\section{Preliminaries}\label{sec-prelims}

\paragraph{Basic notation.} For integers $m$ and $n$, we let $[m,n]$ denote the 
set $\{m,m+1,\dotsc,n\}$. We will also use the notation $[m,n)$ for the set 
$[m,n-1]$,  $(m,n]$ for the set $[m+1,n]$, and $[n]$ for $[1,n]$. We will avoid 
using $(m,n)$ for $[m+1,n-1]$, however; instead, we will use the 
notation $(m,n)$ to denote ordered pairs of integers.

We write $\Mat$ for the set of binary matrices with 
$m$ rows and $n$ columns. We will always assume that rows of matrices are 
numbered top-to-bottom starting with 1, that is, the first row is the topmost.

For a matrix $M\in\Mat$, we let $M(i,j)$ denote the value of the entry 
in row $i$ and column $j$ of~$M$. We say that the pair $(i,j)$ is a 
\emph{1-entry} of $M$ if $M(i,j)=1$, otherwise it is a \emph{0-entry}. The set 
of 1-entries of a matrix $M\in\Mat$ is called the \emph{support} of $M$, denoted 
by $\supp(M)$; formally, $\supp(M)=\{(i,j)\in[m]\times[n];\;M(i,j)=1\}$. 

We say that a binary matrix $M'$ \emph{dominates} a binary matrix $M$, if the 
two matrices have the same number of rows and the same number of columns, and 
moreover, $\supp(M)\subseteq\supp(M')$. In other words, $M$ can be obtained 
from $M'$ by changing some 1-entries into 0-entries.

For a matrix $M\in\Mat$ and for a set of row-indices $R\subseteq[m]$ and 
column-indices $C\subseteq[n]$, we let $M[R\times C]$ denote the submatrix of 
$M$ induced by the rows in $R$ and columns in~$C$. More formally, if 
$R=\{r_1<r_2<\dotsb<r_k\}$ and $C=\{c_1<c_2<\dotsb<c_\ell\}$, then $M[R\times 
C]$ is a matrix $P\in\Pat$ such that $P(i,j)=M(r_i,c_j)$ for every 
$(i,j)\in[k]\times[\ell]$. 

A \emph{line} in a matrix $M$ is either a row or a column of $M$. We 
view a line as a special case of a submatrix. For instance, the $i$-th row of 
a matrix $M\in\Mat$ is the submatrix $M[\{i\}\times[n]]$. A \emph{horizontal 
interval} is a submatrix formed by consecutive entries belonging to a single 
row, i.e., a submatrix of the form $M[\{i\}\times[j_1,j_2]\}$ where $i$ is a 
row index and $j_1,j_2$ are column indices. Vertical intervals are defined 
analogously.

We say that a submatrix of $M$ is \emph{empty} if it does not contain any 
1-entries.

For a matrix $M\in\Mat$ and an entry $e\in[m]\times[n]$, we let $M\Delta e$ 
denote the matrix obtained from $M$ by changing the value of the entry $e$ from 
0 to 1 or from 1 to~0.

\paragraph{Interval minors.} A \emph{row contraction} in a matrix $M\in\Mat$ is 
an operation that replaces a pair of adjacent rows $r$ and $r+1$  by a single 
row, so that the new row contains a 1-entry in a column $j$ if and only if at 
least 
one of the two original rows contained a 1-entry in column~$j$. Formally, 
the row contraction transforms $M$ into a matrix $M'\in\bin^{(m-1)\times n}$ 
whose entries are defined by
\[
 M'(i,j)=\begin{cases}
          M(i,j)\text{ if }i<r,\\
\max\{M(r,j),M(r+1,j)\}\text{ if }i=r,\\
M(i+1,j)\text{ if }i>r.
         \end{cases}
\]
A column contraction is defined analogously.

We say that a matrix $P\in\Pat$ is an \emph{interval minor} of a matrix 
$M\in\Mat$, denoted $P\im M$, if we can transform $M$ by a sequence of row 
contractions and column contractions to a matrix $P'\in\Pat$ that dominates~$P$.
When $P$ is an interval minor of $M$, we also say that $M$ \emph{contains} $P$, 
otherwise we say that $M$ \emph{avoids} $P$, or $M$ is $P$-avoiding.

\begin{figure}
 \centerline{\includegraphics{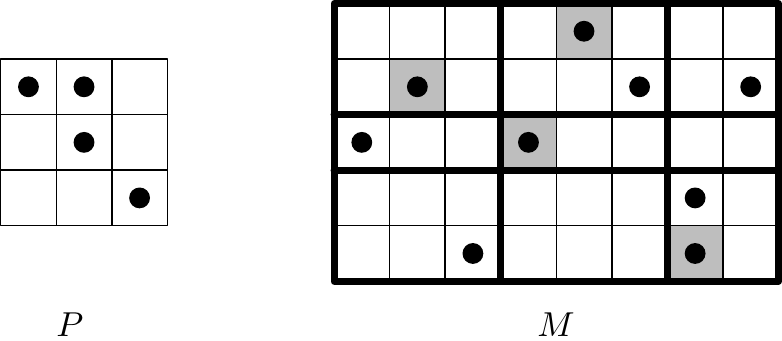}}
\caption{A pattern $P$ and a matrix $M$ that contains $P$. 
The thick lines indicate a partition of $M$ containing $P$, and the shaded 
1-entries form an image of~$P$.}\label{fig-contain}
\end{figure}

There are several alternative ways to define interval minors. One possible 
approach uses the concept of matrix partition. For $P\in\Pat$ and $M\in\Mat$, a 
\emph{partition of $M$ containing $P$} is the sequence of row indices 
$r_0,r_1,\dotsc,r_k$ and column indices $c_0,c_1,\dotsc,c_\ell$ with 
$0\le r_0<r_1<\dotsb<r_k\le m$ and $0\le c_0<c_1<\dotsb<c_\ell\le n$, such that 
for every 1-entry $(i,j)$ of $P$, the submatrix 
$M[(r_{i-1},r_i]\times(c_{j-1},c_j]]$ has at least one 1-entry. See 
Figure~\ref{fig-contain}.

An \emph{embedding} of a matrix $P\in\Pat$ into a matrix $M\in\Mat$ is a 
function $\phi\colon [k]\times[\ell]\to[m]\times[n]$ with the following 
properties:
\begin{itemize}
\item If $e=(i,j)$ is a 1-entry of $P$, then $\phi(e)$ is a 1-entry of 
$M$.
\item Let $e_1=(i_1,j_1)$ and $e_2=(i_2,j_2)$ be two entries of $P$, and 
suppose that $\phi(e_1)=(i^*_1,j^*_1)$ and $\phi(e_2)=(i^*_2,j^*_2)$. If 
$i_1<i_2$ then 
$i^*_1<i^*_2$, and if $j_1<j_2$ then $j^*_1<j^*_2$.
\end{itemize}

Notice that in an embedding $\phi$ of $P$ into $M$, two entries of $P$ 
belonging to the same row may be mapped to different rows of $M$, and similarly 
for columns.

In practice, it is often inconvenient and unnecessary to specify completely an 
embedding of $P$ into~$M$. In particular, it is usually unnecessary to specify 
the image of all the 0-entries in~$P$. This motivates the notion of partial 
embedding, which we now formalise. Consider again binary matrices $P\in\Pat$ and 
$M\in\Mat$. Let $S$ be a nonempty subset of $[k]\times[\ell]$. We say that a 
function $\psi\colon S\to[m]\times[n]$ is a \emph{partial embedding} of $P$ into 
$M$ if the following holds:
\begin{itemize}
\item If $e=(i,j)$ is a 1-entry of $P$, then $e$ is in $S$ and $\psi(e)$ is a 
1-entry of~$M$.
\item An entry $e=(i,j)\in S$ is mapped by $\psi$ to an entry   
$\psi(e)=(i^*,j^*)$ of $M$ satisfying the following inequalities: $i\le i^*$, 
$j\le j^*$, $k-i\le m-i^*$ and $\ell-j\le n-j^*$. Informally, the entry 
$\psi(e)$ is at least as far from the top, left, bottom and right edge of the 
corresponding matrix as the entry~$e$.
\item Let $e_1=(i_1,j_1)$ and $e_2=(i_2,j_2)$ be two entries in $S$, with 
$\psi(e_1)=(i^*_1,j^*_1)$ and $\psi(e_2)=(i^*_2,j^*_2)$. If $i_1< i_2$ then 
$i_2-i_1\le 
i^*_2-i^*_1$, and if $j_1<j_2$ then $j_2-j_1\le j^*_2-j^*_1$.
\end{itemize}

For a partial embedding $\psi$ of a pattern $P$ into a matrix $M$, the 
\emph{image of $P$} (with respect to $\psi$) is the set of entries 
$\{\psi(e);\; e\in\supp(P)\}$ in the matrix~$M$. Note that all the entries in 
the image of $P$ are 1-entries.

\begin{lemma}\label{lem-minor}
For matrices $P\in\Pat$ and $M\in\Mat$ the following properties are equivalent.
\begin{itemize}
\item[1.] $P$ is an interval minor of $M$.
\item[2.] $M$ has a partition containing $P$.
\item[3.] $P$ has an embedding into $M$.
\item[4.] $P$ has a partial embedding into $M$.
\end{itemize}
\end{lemma}
\begin{proof}
We will prove the implications $2\implies 1\implies 3\implies 4\implies 2$.

To see that 2  implies 1, suppose $M$ has a partition containing $P$, 
determined by row indices $r_0,r_1,\dotsc,r_k$ and column indices 
$c_0,c_1,\dotsc,c_\ell$, where we may assume that $r_0=c_0=0$, $r_k=m$ and 
$c_\ell=n$. We may then contract the rows from each interval of 
the form $(r_{i-1},r_i]$ into a single row, and contract the columns from each 
interval $(c_{i-1},c_i]$ to a single column, to obtain a matrix $P'\in\Pat$ 
that dominates~$P$.

To see that 1 implies 3, suppose that $P$ is an interval minor of $M$. This 
means that there is a sequence of matrices $M_0, M_1, M_2,\dotsc, M_s$ with 
$s=(m-k)+(n-\ell)$, where $M_0\in\Pat$ is a matrix that dominates $P$, and for 
each $i\in[s]$, the matrix $M_{i-1}$ can be obtained from $M_i$ by contracting 
a pair of adjacent rows or columns. We can then easily observe that for every 
$i=0,1,\dotsc,s$ there is an embedding $\phi_i$ of $P$ into $M_i$. Indeed, 
reasoning by induction, the embedding $\phi_0$ is the identity map, and for a 
given $i\in[s]$, if there is an embedding $\phi_{i-1}$ of $P$ into $M_{i-1}$, 
then an embedding $\phi_i$ can be obtained by an obvious modification of 
$\phi_{i-1}$.

Clearly, 3 implies 4, since every embedding is also a partial embedding.

To show that 4 implies 2, assume that $\psi\colon S\to[m]\times[n]$ is a 
partial embedding of $P$ into~$M$. We will define a sequence of row indices 
$0\le r_0<r_1<\dotsb<r_k\le m$ with these two properties:
\begin{itemize}
 \item For each entry $e\in S$ that belongs to row $i$ of $P$, the entry 
$\psi(e)$ belongs to a row $i^*$ of $M$ for some $i^*\in(r_{i-1},r_i]$.
\item If $S$ contains at least one entry from row $i$ in $P$, then $S$ contains 
an entry $e$ in row $i$ such that $\psi(e)$ is in row $r_i$ of~$M$.
\end{itemize}
We define the numbers $r_i$ inductively, starting with $r_0=0$. Suppose that 
$r_0,\dotsc,r_{i-1}$ have been defined, for some $i\ge 1$. If $S$ contains no 
entry from row $i$ of $P$, define $r_i=r_{i-1}+1$. On the other hand, if $S$ 
contains an entry from row~$i$, we let $r_i$ be the largest row index of $M$ 
such that $\psi$ maps an entry from row $i$ of $P$ to an entry in row~$r_i$ 
of~$M$. Notice that any entry $e\in S$ that does not belong to the first $i$ 
rows of $P$ must be mapped by $\psi$ to an entry strictly below row $r_i$ of 
$M$, otherwise $\psi$ would not satisfy the properties of a partial embedding. 

In an analogous way, we also define a sequence of column indices $0\le 
c_0<c_1<\dotsb<c_{\ell}\le n$. These sequences will satisfy that for every 
$e=(i,j)\in S$ we have $\psi(e)\in (r_{i-1},r_i]\times(c_{j-1},c_j]$. Since 
$\psi$ is a partial embedding, $S$ contains all the 1-entries of $P$, and 
$\psi$ maps these 1-entries to 1-entries of~$M$. In particular, the sequences 
$(r_i)_{i=0}^k$ and $(c_j)_{j=0}^\ell$ form a partition of $M$ containing~$P$.
\end{proof}

\paragraph{Minor-closed classes.} 
For a matrix $P$, we let $\Avm{P}$ denote the set of all binary matrices that 
do not contain $P$ as an interval minor. We call the matrices in $\Avm{P}$ the 
\emph{avoiders} of~$P$, or \emph{$P$-avoiders}. 

More generally, if $\class{F}$ is a set of matrices, we let $\Avm{\class{F}}$ 
denote the set of binary matrices that avoid all elements of $\class{F}$ as 
interval minors.

We call a set $\class{C}$ of binary matrices a \emph{minor-closed class} (or 
just \emph{class}, for short) if for every matrix $M\in\class{C}$, all the 
interval minors of $M$ are in $\class{C}$ as well. Clearly, $\Avm{\class{F}}$ 
is a class, and for every class $\class{C}$ there is a (possibly infinite) set 
$\class{F}$ such that $\class{C}=\Avm{\class{F}}$. A \emph{principal class} is 
a class of matrices determined by a single forbidden pattern, i.e., a class of 
the form $\Avm{P}$ for a matrix~$P$.

For a class $\class{C}$ of matrices, we say that a matrix $M\in\class{C}$ is 
\emph{critical for $\class{C}$} if the change of any 0-entry of $M$ to a 
1-entry creates a matrix that does not belong to~$\class{C}$. In other words, 
$M\in\class{C}$ is critical for $\class{C}$ if it is not dominated by any other 
matrix in~$\class{C}$. For a pattern $P$, we let $\Avmax{P}$ be the set of 
critical matrices for~$\Avm{P}$, and similarly for a set of patterns $\cF$, 
$\Avmax{\class{F}}$ is the set of all critical matrices for~$\Avm{\class{F}}$.

\subsection{Simple examples of $P$-avoiders}\label{ssec-example}

We conclude this section by presenting several examples of avoiders of 
certain simple patterns. These examples will play a role in 
Section~\ref{sec-bounded}, in the proof of our main result.
We begin with a very simple example, which we present without proof.

\begin{obs}\label{obs-Rk}
Let $R_k$ be the matrix with 1 row and $k$ columns, whose every entry is a 
1-entry (see Figure~\ref{fig-rkdk}). A matrix $M\in\Mat$ avoids $R_k$ if and 
only 
if $M$ has at most $k-1$ 
nonempty columns. Consequently, $M$ is a critical $R_k$-avoider if and only if 
$\supp(M)$ is a union of $\min\{k-1,n\}$ columns.
\end{obs}

\begin{figure}
\centerline{\includegraphics{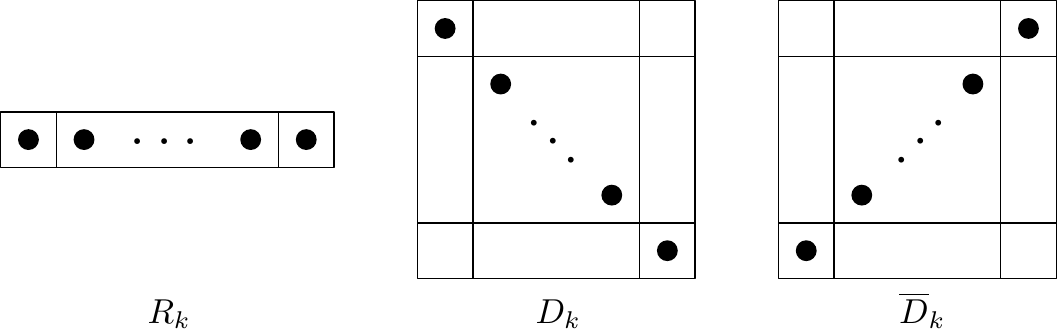}}
 \caption{The patterns $R_k$, $D_k$ and $\ovD_k$.}\label{fig-rkdk}
\end{figure}

Next, we will consider the diagonal patterns $D_k\in\bin^{k\times k}$, defined 
by $\supp(D_k)=\{(i,i);\; i\in[k]\}$, and their mirror image $\ovD_k\in 
\bin^{k\times k}$, defined by $\supp(\ovD_k)=\{(i,k-i+1);\; i\in[k]\}$ (see 
again Figure~\ref{fig-rkdk}). To describe the avoiders of these patterns, we 
first introduce some terminology.

Let $e$ and $e'$ be two entries of a matrix~$M$.
An \emph{increasing walk} from $e$ to $e'$ in $M$ is a set of entries 
$W=\{e_i=(r_i,c_i);\; i=0,\dotsc,t\}$ such that $e_0=e$, $e_t=e'$, and for 
every $i\in[t]$ we have either $r_i=r_{i-1}$ and $c_i=c_{i-1}+1$ (that is, 
$e_i$ is to the right of $e_{i-1}$), or $r_i=r_{i-1}-1$ and $c_i=c_{i-1}$ (that 
is, $e_i$ is above $e_{i-1}$). A \emph{decreasing walk} is defined analogously, 
except now $e_i$ is either to the right or below~$e_{i-1}$.

We say a matrix~$M$ is an \emph{increasing matrix} if $\supp(M)$ is a subset of
an increasing walk. A \emph{decreasing matrix} is defined analogously. See 
Figure~\ref{fig-walk}.

\begin{figure}
 \centerline{\includegraphics[width=0.9\textwidth]{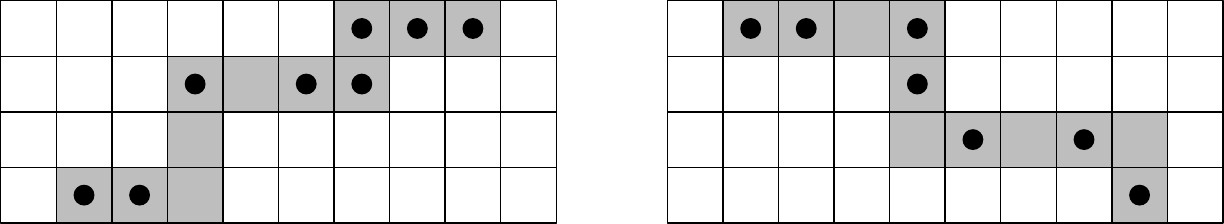}}
\caption{An increasing matrix (left) and a decreasing matrix (right). The 
shaded entries form an increasing and a decreasing walk in the respective 
matrices.}\label{fig-walk}
\end{figure}

\begin{prop}\label{pro-diag}
A matrix $M\in\Mat$ avoids the pattern $D_k$ if and only if $M$ contains $k-1$ 
increasing walks $W_1,\dotsc,W_{k-1}$ from $(m,1)$ to $(1,n)$ such that 
\[\supp(M)\subseteq W_1\cup W_2\cup\dotsb\cup W_{k-1}.\]
\end{prop}
\begin{proof}
Clearly, if $M$ contains $D_k$, then $M$ has $k$ 1-entries no two of which can 
belong to a single increasing walk, and therefore $\supp(M)$ cannot be covered 
by $k-1$ increasing walks.

Suppose now that $M$ avoids~$D_k$. Consider a partial order $\triangleleft$ on 
the set $\supp(M)$, defined as $(i,j)\triangleleft(i',j')\iff i<i'$ and $j<j'$. 
Since $M$ avoids $D_k$, this order has no chain of length~$k$. By the classical 
Dilworth theorem~\cite{dilworth}, $\supp(M)$ is a union of $k-1$ antichains 
of~$\triangleleft$. We may easily observe that each antichain of $\triangleleft$ 
is contained in an increasing walk.
\end{proof}
Proposition~\ref{pro-diag} shows, in particular, that a matrix $M$ avoids the 
pattern $D_2=\smm{\bullet& \\ &\bullet}$ if and only if $M$ is an increasing 
matrix. By symmetry, $M$ avoids $\ovD_2$ if and only if it is a decreasing 
matrix.

Another direct consequence of the proposition is the following 
corollary, describing the structure of critical $D_k$-avoiders. 
\begin{cor}\label{cor-diag}
A critical $D_k$-avoiding matrix $M$ contains $k-1$ increasing  
walks $W_1,\dotsc,W_{k-1}$ from $(m,1)$ to $(1,n)$ such that 
 $\supp(M)= W_1\cup W_2\cup\dotsb\cup W_{k-1}$.  
\end{cor}
Note that Corollary~\ref{cor-diag} only gives a necessary condition for a matrix 
to be a critical $D_k$-avoider, therefore it is not a characterisation of 
critical $D_k$-avoiders. With only a little bit of extra effort, we could state 
and prove such a characterisation, but we omit doing so, as we do not need it 
for our purposes.

A simple but useful observation is that adding an empty row or column to the 
boundary of a pattern affects the $P$-avoiders in a predictable way. We state 
it here without proof.

\begin{obs}\label{obs-empty} Let $P\in\Pat$ be a pattern, and let 
$P'\in\bin^{k\times(\ell+1)}$ be the pattern obtained by appending an 
empty column to~$P$; in other words, we have $P'[[k]\times[\ell]]=P$, and the 
last column of $P'$ is empty. Then a matrix $M\in\Mat$ avoids $P'$ if and only 
if the matrix obtained by removing the last column from $M$ avoids~$P$. 
Consequently, $M$ is in $\Avmax{P'}$ if and only if all the entries in the last 
column of $M$ are 1-entries, and the preceding columns form a matrix from 
$\Avmax{P}$. Analogous properties hold for a pattern $P''$ obtained by 
prepending an empty column in front of all the columns of $P$, and also for rows 
instead of columns.
\end{obs}

\section{Line complexity}
\label{sec-bounded}
In the previous section, we have seen several examples of matrices avoiding a 
fixed pattern as interval minor. At a glance, it is clear that these matrices 
are highly structured. We would now like to make the notion of `highly 
structured matrices' rigorous, and generalize it to other forbidden patterns.

We will focus on the local structure of matrices, i.e., the structure observed 
by looking at a single row or column. For a forbidden pattern $P$ with at least 
two rows and two columns, it is not hard to see that any binary vector can 
appear as a row or column of a $P$-avoiding matrix. 

However, the situation changes when we restrict our attention to critical 
$P$-avoiders. In the examples of critical $P$-avoiders we saw in 
Subsection~\ref{ssec-example}, the 1-entries in each row or column were 
clustered into a bounded number of intervals. In particular, for these patterns 
$P$, there are only at most polynomially many vectors of a given length that may 
appear as rows or columns of a critical $P$-avoider.

In this section, we study this phenomenon in detail. We show that it generalizes 
to many other forbidden patterns $P$, but not all of them. As our main result, 
we will present a complete characterisation of the patterns $P$ exhibiting this 
phenomenon.

Let us begin by formalising our main concepts.

A \emph{horizontal 0-run} in a matrix $M$ is a maximal sequence of consecutive 
0-entries in a single row. More formally, a horizontal interval 
$M[\{r\}\times[c_1,c_2]]$ is a \emph{horizontal 0-run} if all its entries are 
0-entries, $c_1=1$ or $M(r,c_1-1)=1$, and $c_2=n$ or $M(r,c_2+1)=1$. 
Symmetrically, a vertical interval is a \emph{vertical 0-run} if it is a maximal 
vertical interval that only contains 0-entries. In the same manner, we define a 
(horizontal or vertical) \emph{1-run} to be a maximal interval of consecutive 
1-entries in a single line of~$M$.

Note that each line in a matrix $M$ can be uniquely decomposed into an 
alternating sequence of 0-runs and 1-runs.

Let $M$ be a binary matrix. The \emph{complexity} of a line of $M$ is the 
number of 0-runs contained in this line. The \emph{row-complexity} of $M$ is 
the maximum complexity of a row of $M$, i.e., the least number $k$ such that 
each row has complexity at most~$k$. Similarly, the \emph{column-complexity} of 
$M$ is the maximum complexity of a column of~$M$.

For a class of matrices $\class{C}$, we define its \emph{row-complexity}, 
denoted $\rc{\class{C}}$, as the supremum of the row-complexities of the 
critical matrices in~$\class{C}$. We say that $\class{C}$ is \emph{row-bounded} 
if $\rc{\class{C}}$ is finite, and \emph{row-unbounded} otherwise. 
Symmetrically, we define the \emph{column-complexity}~$\cc{\class{C}}$ of 
$\class{C}$ and the property of being \emph{column-bounded} and 
\emph{column-unbounded}. We say that a class $\class{C}$ is \emph{bounded} if it 
is both row-bounded and column-bounded; otherwise, it is \emph{unbounded}.

We stress that when defining the row-complexity and column-complexity of a class 
of matrices, we only take into account the matrices that are critical for the 
class. 

We are now ready to state our main result.

\begin{thm}\label{thm-main} Let $P$ be a pattern.
The class $\Avm{P}$ is row-bounded if and only if $P$ does not contain any of 
$Q_1, Q_2, Q_3, Q_4$ as an interval minor, where
\[
Q_1=\smm{ &\bullet& \\\bullet& & \\ & &\bullet},\ Q_2=\smm{ &\bullet& \\ & 
&\bullet\\\bullet& & },\ Q_3=\smm{\bullet& & \\ & &\bullet\\ &\bullet& 
}\text{ and } 
Q_4=\smm{ & &\bullet\\\bullet& & \\ &\bullet& }.
\]
\end{thm}

Before we prove Theorem~\ref{thm-main}, we point out two of its 
direct consequences.

\begin{cor}\label{cor-first}
For a pattern $P$, these statements are equivalent:
\begin{itemize}
\item $\Avm{P}$ is row-bounded.
\item $\Avm{P}$ is column-bounded.
\item $\Avm{P}$ is bounded.
\end{itemize}
\end{cor}

\begin{cor}\label{cor-second}
Let $\cC=\Avm{P}$ and $\cC'=\Avm{P'}$ be principal classes, and suppose that
$\cC\subseteq\cC'$ (or equivalently, $P\im P'$). If 
$\cC'$ is bounded, then $\cC$ is bounded as well.
\end{cor}

Although each of these two corollaries is stating a seemingly basic property
of the boundedness notion, we are not able to prove either of them without 
first proving Theorem~\ref{thm-main}. We also remark that neither of the two 
corollaries can be generalized to non-principal classes of matrices, as we will 
see in Section~\ref{sec-further}.

Let us say that a pattern $P$ is \emph{row-bounding} if $\Avm{P}$ is 
row-bounded, otherwise $P$ is \emph{non-row-bounding}. Similarly, $P$ is 
\emph{bounding} if $\Avm{P}$ is bounded and \emph{non-bounding} otherwise.

Let $\rbpat$ be the set of patterns $\{Q_1,Q_2,Q_3,Q_4\}$. 
Theorem~\ref{thm-main} states that a pattern $P$ is row-bounding if and only if 
$P$ is in~$\Avm{\rbpat}$. To prove this, we will proceed in several steps. We 
first show, in Subsection~\ref{ssec-nonbound}, that if $P$ contains a pattern 
from $\rbpat$, then $P$ is not row-bounding. This is the easier part of the 
proof, though by no means trivial. Next, in Subsection~\ref{ssec-bound}, we 
show that every pattern in $\Avm{\rbpat}$ is row-bounding. This part is 
more technical, and requires a characterisation the structure of the 
patterns in~$\Avm{\rbpat}$.

\subsection{Non-row-bounding patterns}
\label{ssec-nonbound}
Our goal in this subsection is to show that any pattern $P$ that contains one of 
the matrices from $\rbpat$ is not row-bounding. Let us therefore fix such a 
pattern~$P$. Without loss of generality, we may assume that $Q_1\im P$.

\begin{thm}
\label{thm-unbound}
For every matrix $P$ such that $Q_1\im P$, the class $\Avm{P}$ is 
row-unbounded.
\end{thm}
\begin{proof}
Refer to Figure~\ref{fig-manyints}.
Let $P\in\Pat$ be a pattern containing $Q_1$ as an interval minor. In 
particular, there are row indices $1\le r_1<r_2<r_3\le k$ and column indices 
$1\le c_1<c_2<c_3\le \ell$ such that $P(r_1,c_2)=P(r_2,c_1)=P(r_3,c_3)=1$.

For an arbitrary integer $p$, we will show how to construct a matrix in 
$\Avmax{P}$ of row-complexity at least~$p$. We first describe a matrix 
$M\in\Mat$ with $m=r_1+p(r_3-r_1)+(k-r_3)$ and 
$n=(c_1-1)+p(c_3-c_1+1)+(\ell-c_3)$.

In the matrix $M$, the leftmost $c_1-1$ columns, the rightmost $\ell-c_3$ 
columns, the topmost $r_1-1$ rows and the bottommost $k-r_3$ rows have all 
entries equal to~1. We call these entries the \emph{frame} of~$M$.

In the $r_1$-th row of $M$, there are $p$ 0-entries appearing in columns 
$c_2+i(c_3-c_1+1)$ for $i=0,\dotsc,p-1$, and the remaining entries in row $r_1$ 
are 1-entries.

The remaining entries of $M$, that is, the entries in rows 
$r_1+1,\dotsc,m-(k-r_3)$ and columns $c_1,\dotsc,n-(\ell-c_3)$, form a 
submatrix with $p(r_3-r_1)$ rows and $p(c_3-c_1+1)$ columns. We partition these 
entries into rectangular blocks, each block with $r_3-r_1$ rows and 
$c_3-c_1+1$ columns. For $i,j\in\{0,\dotsc,p-1\}$, let $B_{i,j}$ be such a 
block, with top-left corner in row $r_1+1+i(r_3-r_1)$ and column 
$c_1+j(c_3-c_1+1)$. The entries in $B_{i,j}$ are all equal to 1 if $i+j=p-1$, 
otherwise they are all equal to 0.

We claim that the matrix $M$ avoids~$P$. To see this, assume there is an 
embedding $\phi$ of $P$ into $M$, and consider where $\phi$ maps the three 
1-entries $e_1=(r_1,c_2)$, $e_2=(r_2,c_1)$, and $e_3=(r_3,c_3)$. Note that none 
of these three entries can be mapped into the frame of~$M$, and moreover, 
neither $e_2$ nor $e_3$ can be mapped to the $r_1$-th row of~$M$. In 
particular, $\phi(e_3)$ is inside a block $B_{i,j}$ for some $i+j=p-1$. Since 
$\phi(e_2)$ is to the top-left of $\phi(e_3)$, it must belong to the same block 
$B_{i,j}$. It follows that $\phi(e_2)$ is in the leftmost column of $B_{i,j}$, 
which is the column $c_1+j(c_3-c_1+1)$, and $\phi(e_3)$ in its rightmost 
column, i.e., the column $c_3+j(c_3-c_1+1)$. Therefore, $\phi(e_1)$ is in 
column $c_2+j(c_3-c_1+1)$; however, all the entries in this column where $\phi$ 
could map $e_1$ are 0-entries. Therefore $M$ is in~$\Avm{P}$.

The matrix $M$ is not necessarily a critical $P$-avoider. However, we 
can transform it into a critical $P$-avoider by greedily changing 0-entries to 
1-entries as long as the resulting matrix stays in~$\Avm{P}$. By this process, 
we obtain a matrix $M'\in\Avmax{P}$ that dominates~$M$. We claim that the 
$r_1$-th row of 
$M'$ is the same as the $r_1$-th row of $M$. This is because changing any 
0-entry in the $r_1$-th row of $M$ to a 1-entry produces a matrix containing 
the complete pattern $1^{k\times\ell}$ as a submatrix, and in particular also 
containing~$P$ as a minor.

We conclude that the matrix $M'\in\Avmax{P}$ has row-complexity at least $p$, 
showing that $\Avm{P}$ is indeed row-unbounded.
\end{proof}

\begin{figure}[!ht]
\includegraphics[width=\textwidth]{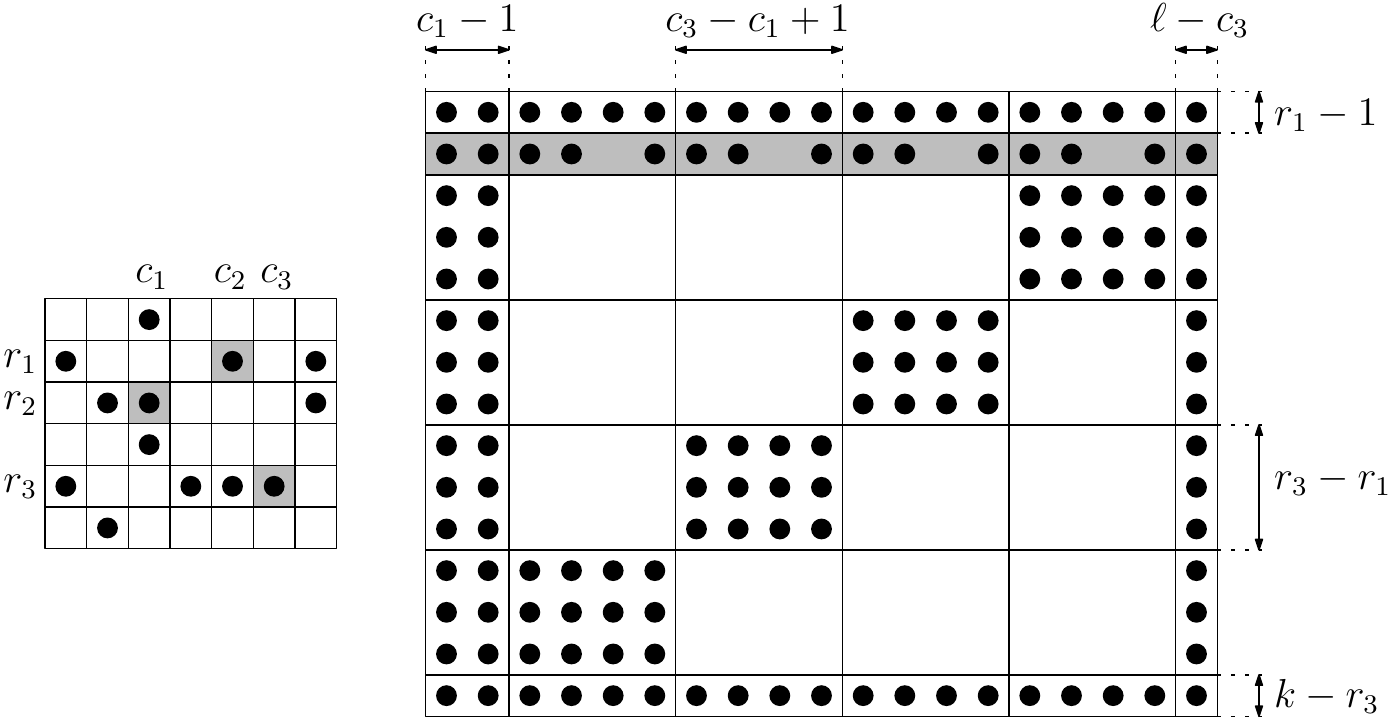}
\caption{Illustration of the proof of Theorem~\ref{thm-unbound}. Left: a 
pattern $P$ with a shaded image of~$Q_1$. Right: a $P$-avoider with a 
shaded row of complexity~$p=4$.}
\label{fig-manyints}
\end{figure}

\subsection{Row-bounding patterns}
\label{ssec-bound}

We now prove the second implication of Theorem~\ref{thm-main}, that is, we 
show that any pattern $P$ avoiding the four patterns in $\rbpat$ is 
row-bounding (and therefore, by symmetry, also column-bounding). We first 
prove a result describing the structure of the patterns $P\in\Avm{\rbpat}$.

We say that a matrix~$M$ can be \emph{covered by $k$ lines} if there is a set 
of lines~$\ell_1,\dots,\ell_k$ such that each 1-entry of $M$ belongs to 
some~$\ell_i$. The following fact is a version of the classical K\H 
onig--Egerváry theorem. We present it here without proof; a proof can be found, 
e.g., in Kung~\cite{kung}.

\begin{fct}[K\H onig--Egerváry theorem]
\label{fac-eger}
A matrix~$M$ cannot be covered by $k$ lines if and only if $M$ contains a set 
of $k+1$ 1-entries, no two of which are in the same row or column.
\end{fct}

\begin{prop}
\label{pro:boundedints}
If a pattern $P$ belongs to $\Avm{\rbpat}$, then 
\begin{enumerate}
\item $P$ avoids the pattern $D_2=\smm{\bullet&\\&\bullet}$, or
\item $P$ avoids the pattern $\ovD_2=\smm{&\bullet\\ \bullet&}$, or
\item $P$ can be covered by three lines.
\end{enumerate}
\end{prop}
\begin{proof}
Assume $P$ cannot be covered by three lines. By Fact~\ref{fac-eger}, 
$P$ contains four 1-entries $e_1=(r_1,c_1)$, $e_2=(r_2,c_2)$, $e_3=(r_3,c_3)$ 
and $e_4=(r_4,c_4)$, no two of which are in the same row or column. We 
may assume that $r_1<r_2<r_3<r_4$. Moreover, since $P$ does not contain any 
pattern from $\rbpat$, we see that any three entries among $e_1, e_2, e_3, e_4$
must form an image of $D_3$ or of~$\ovD_3$. Consequently, the four entries 
$e_i$ form an image of $D_4$ or of~$\ovD_4$, i.e., we must have either 
$c_1<c_2<c_3<c_4$ or $c_1>c_2>c_3>c_4$. Suppose that $c_1<c_2<c_3<c_4$ holds, 
the other case being symmetric. 

\begin{figure}
 \centerline{\includegraphics{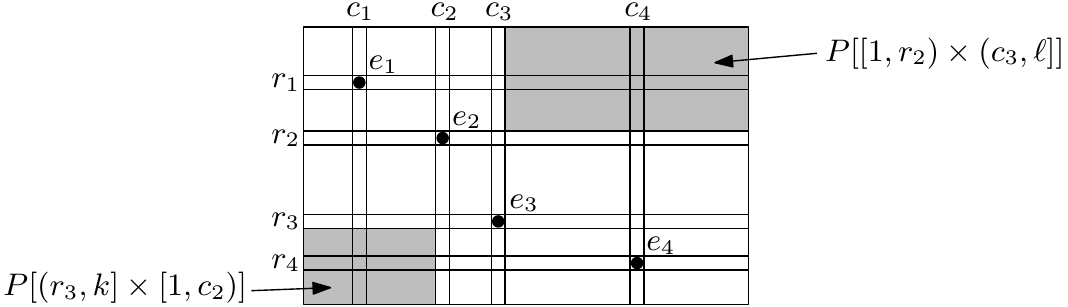}}
\caption{Illustration of the proof of 
Proposition~\ref{pro:boundedints}.}\label{fig-boundedints}
\end{figure}

We will now show that $P$ avoids the pattern $\ovD_2$. 
Note first that the submatrix $P[[r_3]\times[c_3]]$ avoids $\ovD_2$, since an 
image of $\ovD_2$ there would form an image of $Q_1$ with~$e_4$. Therefore, by 
Proposition~\ref{pro-diag}, all the 1-entries in $P[[r_3]\times[c_3]]$ belong to 
a single decreasing walk from $(1,1)$ to~$e_3$. Symmetrically, all 1-entries in 
the submatrix $P[[r_2,k]\times[c_2,\ell]]$ belong to a decreasing walk from 
$e_2$ to $(k,\ell)$. 

Moreover, there can be no 1-entry in $P[(r_3,k]\times [1,c_2)]$ or in 
$P[[1,r_2)\times(c_3,\ell]]$, since such a 1-entry would form a forbidden 
pattern with $e_2$ and~$e_3$. We conclude that all the 1-entries of $P$ belong 
to a single decreasing walk and therefore $P$ avoids~$\ovD_2$.
\end{proof}

We note that Proposition~\ref{pro:boundedints} is not an equivalent 
characterisation of patterns from $\Avm{\rbpat}$, since a matrix covered by 
three lines may contain a pattern from~$\rbpat$. Later, in 
Lemma~\ref{lem-2types}, we will give a more precise description of the avoiders 
of $\rbpat$ that cannot be covered by two lines.

\paragraph{Relative row-boundedness.} Before we prove that each pattern $P$ in 
the set $\Avm{\rbpat}$ is row-bounding, we need some technical preparation. 
First of all, we shall need a more refined notion of row-boundedness, which 
considers individual 1-entries of the pattern $P$ separately.

Let $P$ be a pattern, let $e$ be a 1-entry of $P$, let $M$ be a $P$-avoiding 
matrix, and let $f$ be a 0-entry of~$M$. Recall that $M\Delta f$ is the matrix 
obtained from $M$ by changing the entry $f$ from 0 to 1. We say that the entry 
$f$ of $M$ is \emph{critical for $e$ (with respect to $P$)} if there is an 
embedding of $P$ into $M\Delta f$ that maps $e$ to~$f$. Moreover, if $z$ is a 
0-run in $M$, we say that $z$ is \emph{critical for $e$} if at least one 0-entry 
in $z$ is critical for~$e$.

Note that a $P$-avoiding matrix is critical for $\Avm{P}$ if and only if each 
0-entry of $M$ is critical for at least one 1-entry of~$P$. 

Let $e$ be a 1-entry of a pattern~$P$. Let $M$ be a matrix avoiding $P$. The 
\emph{complexity of a row $r$ of $M$ relative to $e$} is the number of 
0-runs in row $r$ that are critical for~$e$. The \emph{row-complexity of $M$ 
relative to $e$} is the maximum complexity of a row of $M$ relative to 
$e$, and the \emph{row-complexity of $\Avm{P}$ relative to $e$}, denoted 
$\rce{\Avm{P}}{e}$, is the supremum of the row-complexities of the matrices 
in $\Avm{P}$ relative to~$e$. When $\rce{\Avm{P}}{e}$ is finite, we say that 
$\Avm{P}$ is \emph{row-bounded relative to $e$} and \emph{$e$ is row-bounding}, 
otherwise $\Avm{P}$ is \emph{row-unbounded relative to $e$}.

Notice that in the definition of $\rce{\Avm{P}}{e}$, we are taking supremum 
over all the matrices in $\Avm{P}$, not just the critical ones. This makes the 
definition more convenient to work with, but it does not make any substantial
difference. In fact, for a pattern $P$ with a row-bounding 1-entry $e$, the 
row-complexity relative to $e$ in $\Avm{P}$ is maximized by a critical 
$P$-avoider. To see this, suppose that $M$ is a $P$-avoiding matrix, $M^+$ is 
any critical $P$-avoiding matrix that dominates $M$, and $f$ is a 0-entry of 
$M$ that is critical for $e$; then $f$ is necessarily also a 0-entry in~$M^+$, 
and is still critical for~$e$ in~$M^+$. Therefore, the row-complexity of $M^+$ 
relative to $e$ is at least as large as the row-complexity of $M$ relative 
to~$e$. 

Observe that the following inequalities hold for any pattern $P$:
\[
 \max_{e\in\supp(P)} \rce{\Avm{P}}{e} \le \rc{\Avm{P}} \le \sum_{e\in\supp(P)} 
\rce{\Avm{P}}{e}.
\]
In particular, a pattern $P$ is row-bounding if and only if each 1-entry of $P$ 
is row-bounding.

%
%
%

\begin{lemma}\label{lem-columns}
Let $P$ be a pattern, and let $M$ be a $P$-avoiding matrix. Let $z$ be a 
horizontal 0-run of $M$, and let $f\in z$ be a 0-entry in this 0-run. Assume 
that there is an embedding $\phi$ of $P$ into~$M\Delta f$. Then $P$ has a 
1-entry $e$ mapped by $\phi$ to $f$, and moreover, every entry of $P$ in the 
same column as $e$ is mapped by $\phi$ to a column containing an entry from~$z$.
\end{lemma}
\begin{proof}
Clearly, $\phi$ must map a 1-entry of $P$ to the entry $f$, otherwise $\phi$ 
would also be an embedding of $P$ into $M$ and $M$ would not be $P$-avoiding.

Suppose now that $z=\{r\}\times[c_1,c_2]$ for a row $r$ and columns $c_1\le 
c_2$. Let $e'$ be an entry of $P$ in the same column as $e$. Suppose that $\phi$ 
maps $e'$ to an entry in column $c$, with $c\not\in[c_1,c_2]$. Assume that 
$c<c_1$, the case $c>c_2$ being analogous. Then we may modify $\phi$ to map $e$ 
to the 1-entry $(r,c_1-1)$ instead of $f$, obtaining an embedding of $P$ into 
$M$, which is a contradiction.
\end{proof}

\paragraph{Criteria for relative row-boundedness.}
Let us first point out a trivial but useful fact: if $\overline P\in\Pat$ is 
a pattern obtained from a pattern $P$ by reversing the order of rows (i.e., 
turning $P$ upside down) then a 1-entry $e=(i,j)$ of $P$ is row-bounding if and 
only if the corresponding 1-entry $\overline e=(k-i+1,j)$ of $\overline P$ is 
row-bounding. Analogous properties hold for reversing the order of columns or 
180-degree rotation. Similarly, operations that map rows to columns, 
such as transposition or 90-degree rotation, will map row-bounding 1-entries to 
column-bounding ones and vice versa.

We will now state several general criteria for row-boundedness of 1-entries, 
which we will later use to show that any $\rbpat$-avoiding pattern is 
row-bounding.
\begin{lemma}
 \label{lem-leftmost} If $P\in\Pat$ is a pattern with a row $r\in[k]$ and a 
column $c\in[\ell]$ such that $\supp(P)\subseteq (\{r\}\times[\ell]) \cup ( 
[m]\times[c,\ell])$, then every 1-entry of $P$ in the interval $\{r\}\times[c]$ 
is row-bounding (see Figure~\ref{fig-leftmost}).
\end{lemma}
\begin{figure}
\centerline{\includegraphics{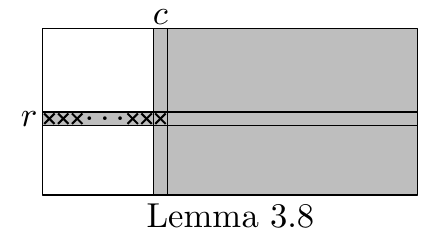}\hskip 10pt\includegraphics{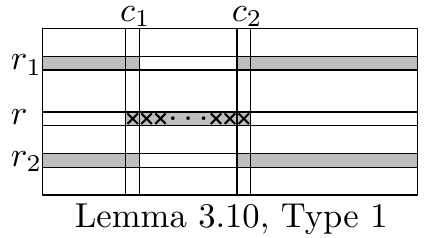}}
\centerline{\includegraphics{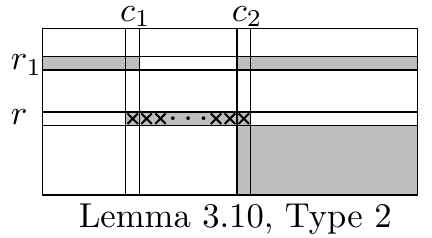}\hskip 10pt\includegraphics{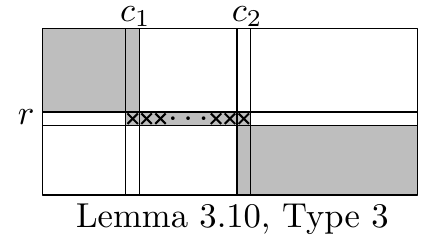}}
 \caption{Illustration of Lemma~\ref{lem-leftmost} and Lemma~\ref{lem-H}. The 
shaded areas are the possible locations of 1-entries. The 1-entries in the
cells marked by crosses are row-bounding.}\label{fig-leftmost}
\end{figure}
\begin{proof}
Let $e=(r,j)$ be a 1-entry of $P$ with $j\le c$. Let $M\in\Mat$ be a  
$P$-avoider, let $f=(r',c')$ be a 0-entry of $M$ critical for~$e$, and let 
$z$ be the horizontal 0-run containing~$f$. 

We claim that in the row $r'$ of $M$, there are fewer than $j$ 1-entries to the 
left of~$f$. Suppose this is not the case, i.e., row $r'$ contains $j$ 
distinct 1-entries $f'_1, f'_2,\dotsc,f'_j$, numbered left to right, all of 
them to the left of~$f$. 

Let $\phi$ be an embedding of $P$ into $M\Delta f$ which maps $e$ to~$f$. 
Recall from Lemma~\ref{lem-columns} that all the entries in column $j$ of $P$ 
are mapped to columns intersecting~$z$. In particular, all the entries from 
column $j$ are mapped to the right of~$f'_j$.

We define a partial embedding $\psi$ of $P$ into $M$, as follows.
Firstly, $\psi$ maps the entries $(r,1), (r,2),\dotsc,(r,j)$ of $P$ to the 
1-entries $f'_1, f'_2,\dotsc,f'_j$ of~$M$. Next, $\psi$ maps each 1-entry
of $P$ that is not among $(r,1), (r,2),\dotsc,(r,j)$ to the same entry 
as~$\phi$. We easily see that $\psi$ is a partial embedding of $P$ into $M$, a 
contradiction.

Therefore, there are fewer than $j$ 1-entries in row $r'$ to the right of $f$,
and hence row $r$ has at most $j$ 0-runs critical for~$e$. Consequently, 
$\rce{\Avm{P}}{e}\le j$ and $e$ is row-bounding.
\end{proof}
The assumptions of Lemma~\ref{lem-leftmost} are satisfied when $c$ is the 
leftmost nonempty column of a pattern $P$ and $r$ is an arbitrary row. We state 
this important special case as a separate corollary.
\begin{cor}\label{cor-leftmost}
Any 1-entry in the leftmost nonempty column of a pattern $P$ is row-bounding.
\end{cor}
\begin{lemma}
\label{lem-H}
Let $P\in\Pat$ be a pattern with a row $r$, and two 
distinct columns $c_1<c_2$, such that all the 1-entries of $P$ in row $r$ 
belong to the interval $\{r\}\times[c_1,c_2]$. Moreover, if $c$ is a column
index with $c_1<c<c_2$, then $P$ has no 1-entry in column $c$ except possibly 
for the entry $(r,c)$. Suppose furthermore that $P$ satisfies one of the 
following three conditions (see Figure~\ref{fig-leftmost}):
\begin{itemize}
 \item[Type 1:] All the 1-entries of $P$ above row $r$ are in a single row 
$r_1<r$, and all the 1-entries below row $r$ are in a single row~$r_2>r$.

\item[Type 2:] All the 1-entries of $P$ above row $r$ are in a single row 
$r_1<r$, and all the 1-entries below row $r$ are in the submatrix 
$P[(r,k]\times[c_2,\ell]]$.

\item[Type 3:] All the 1-entries of $P$ above row $r$ are in the submatrix
$P[[1,r)\times[c_1]]$, and all the 1-entries below row $r$ are in the 
submatrix $P[(r,k]\times[c_2,\ell]]$.
\end{itemize}
Then every 1-entry in the interval $\{r\}\times[c_1,c_2]$ is row-bounding.
\end{lemma}
\begin{proof}
Let~$P\in\Pat$ be a pattern satisfying the assumptions, and let $d=c_2-c_1+1$. 
We will show that for each 1-entry~$e\in \{r_2\}\times[c_1,c_2]$ of $P$ and 
every $P$-avoiding matrix $M\in\Mat$, there are at most $d$ 0-runs critical for 
$e$ in each row of~$M$. 

For contradiction, assume that $M$ has a row~$r'$ with at least $d+1$ 0-runs 
critical for~$e$. Let $f$ and $f'$ be the leftmost and the rightmost 0-entries 
critical for $e$ in row~$r'$. By assumption, $M$ has at least $d$ 
1-entries in row $r'$ between $f$ and~$f'$. Let $f_1,f_2,\dotsc,f_d$ be $d$ 
such 1-entries, numbered left to right.

Let $\phi$ be an embedding of $P$ into $M\Delta f$ which maps $e$ to $f$, and 
let $\phi'$ be an embedding of $P$ into $M\Delta f'$ which maps $e$ to~$f'$.
Let us describe a partial embedding $\psi$ of $P$ into~$M$. Firstly, $\psi$ 
maps the entries $(r,c_1),\allowbreak(r,c_1+1),\allowbreak\dotsc,(r,c_2)$ to 
the entries $f_1,f_2,\dotsc,f_d$ in row~$r'$ of $M$. Next, $\psi$ maps each 
1-entry in $M[[m]\times[c_1]]$ except $(r,c_1)$ to the same entry as $\phi$, and 
$\psi$ maps the 1-entries in $M[[m]\times[c_2,n]]$ except $(r,c_2)$ to the same 
entry as~$\phi'$. We easily check that this makes $\psi$ a partial embedding of 
$P$ into $M$: note that from Lemma~\ref{lem-columns}, it follows that $\phi$ 
maps all the entries in column $c_1$ of $P$ to entries strictly to the left of 
$f_1$, and $\phi'$ maps entries in column $c_2$ to entries strictly to the right 
of~$f_d$. 

This is impossible, since $M$ is $P$-avoiding. Therefore, every row of a 
$P$-avoiding matrix has at most $d$ 0-runs critical for $e$, and $e$ is 
row-bounding.
\end{proof}

\begin{lemma}
\label{lem-I}
Let $P\in\Pat$ be a pattern with two rows $r_1\le r_2$ and a column $c$, 
such that for every $r\in[r_1,r_2]$, $P$ has no 1-entry in row $r$ except 
possibly for the entry $(r,c)$. Suppose moreover, that $P$ satisfies one of the 
following conditions (see Figure~\ref{fig-I}):
\begin{itemize}
\item[Type 1:] All the 1-entries of $P$ above row $r_1$ are in column $c$ 
or in the row $r_1-1$, and all the 1-entries below row $r_2$ are in column $c$ 
or in the row $r_2+1$.
\item[Type 2:] All the 1-entries of $P$ above row $r_1$ are in column $c$ 
or in the row $r_1-1$, and all the 1-entries below row $r_2$ are in the 
submatrix $P[(r_2,k]\times [c,\ell]]$.
\item[Type 3:] All the 1-entries of $P$ above row $r_1$ are in the 
submatrix $P[[1,r_1)\times[c]]$, and all the 1-entries below row $r_2$ are in 
the submatrix $P[(r_2,k]\times [c,\ell]]$.
\end{itemize}
Then every 1-entry in the interval  $[r_1,r_2]\times\{c\}$ is row-bounding.
\end{lemma}
\begin{figure}
\centering
\includegraphics{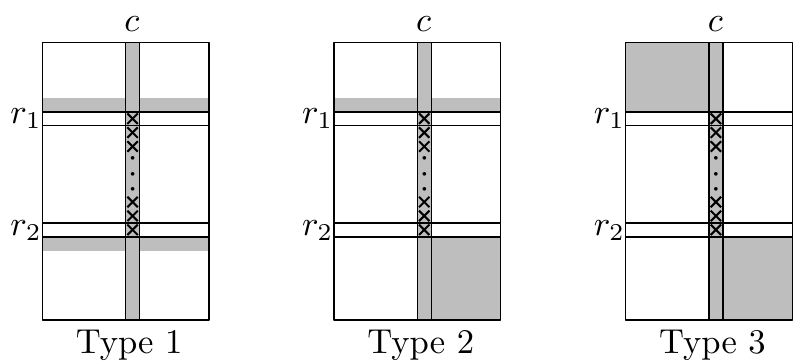}
\caption{Illustration of Lemma~\ref{lem-I}. The shaded areas correspond to 
possible locations of 1-entries. The 1-entries in cells marked by crosses are 
row-bounding.}
\label{fig-I}
\end{figure}
\begin{proof}
Let $P$ be a pattern satisfying the assumptions of the lemma, and let 
$e =(r,c)$ be its 1-entry, with $r\in[r_1,r_2]$. Let $M$ be a  
$P$-avoider. We claim that every row of $M$ has at most one 0-run critical 
for~$e$. For contradiction, suppose that row $i$ of $M$ has two 0-runs $z_L$ 
and $z_R$ critical for $e$, where $z_L$ is to the left of $z_R$. Let $f_L\in 
z_L$ and $f_R\in z_R$ be two 0-entries critical for~$e$ in the two 0-runs.

Let $\phi_L$ be an embedding of $P$ into $M\Delta f_L$ with $\phi_L(e)=f_L$, 
and $\phi_R$ be an embedding mapping $P$ into $M\Delta f_R$ 
with~$\phi_R(e)=f_R$. We will describe a partial embedding $\psi$ of $P$ 
into~$M$.

Since $f_L$ and $f_R$ are in distinct 0-runs, $M$ has a 1-entry $f$ that lies 
in row $i$ between $f_L$ and~$f_R$. We put $\psi(e)=f$. For 
any other 1-entry $e'\in\supp(P)\setminus\{e\}$, we will define $\psi(e')$ to 
be equal to either $\phi_L(e')$ or $\phi_R(e')$, by the following rules.

For a 1-entry $e'$ which is strictly to the left of column $c$, we let 
$\psi(e')=\phi_L(e')$ and for a 1-entry $e'$ strictly to the right of column 
$c$, we let $\psi(e)=\phi_R(e')$. 

It remains to deal with the 1-entries in column~$c$. For a 1-entry $e'$ in 
$[r_1,r)\times\{c\}$, we choose $\psi(e')$ to be the lower of the two entries 
$\phi_L(e')$ and $\phi_R(e')$, i.e., we choose the entry that has larger 
row-index. If $\phi_L(e')$ and $\phi_R(e')$ are in the same row, we choose 
$\psi(e')$ arbitrarily from the two options.

For a 1-entry $e'$ in $[1,r_1)\times\{c\}$, we distinguish two 
possibilities. If $P$ is of Type 1 or Type 2, that is, all 1-entries above row 
$r_1$ are in column $c$ or row $r_1-1$, we choose $\psi(e')$ to be the higher 
of the two entries $\phi_L(e')$ and~$\phi_R(e')$. If, on the other hand, $P$ is 
of Type 3, so all 1-entries above row $r_1$ are in columns $1,\dotsc,c$, we 
put $\psi(e')=\phi_L(e')$.

We proceed symmetrically for 1-entries below row~$r$. For a 1-entry $e'\in 
(r,r_2]\times\{c\}$, we choose $\psi(e')$ to be the higher of the two entries 
$\phi_L(e')$ and $\phi_R(e')$, breaking ties arbitrarily. For a 1-entry $e'\in 
(r_2,k]\times\{c\}$, if $P$ is of Type 1, we let $\psi(e')$ be the lower of 
$\phi_L(e')$ and $\phi_R(e')$, and if $P$ is of Type 2 or 3, we put 
$\psi(e')=\phi_R(e')$.

Note that we may deduce from Lemma~\ref{lem-columns} that $\phi_L$ maps all 
the entries in column $c$ of $P$ to entries strictly to the left of $f$, and 
$\phi_R$ maps entries from column $c$ to entries strictly to the right of~$f$.
We may then easily verify that the mapping $\psi$ is a partial embedding of $P$ 
into~$M$. This contradiction shows that the entry $e=(r,c)$ is row-bounding.
\end{proof}

\begin{lemma}\label{lem-I2} Let $P\in\Pat$ be a pattern with two rows $r_1<r_2$ 
and two columns $c_1<c_2$ of one of the following two types (see 
Figure~\ref{fig-I2}):
\begin{itemize}
\item[Type 1:]$\supp(P)\subseteq ([r_1,r_2]\times\{c_1\})\cup\big(
\{r_1,r_2\}\times ([c_1]\cup\{c_2\})\big)$.
\item[Type 2:]$\supp(P)\subseteq ([r_1,r_2]\times\{c_1\})\cup\big(
\{r_2\}\times 
([c_1]\cup\{c_2\})\big)\cup([r_1]\times\{c_2\})$.          
\end{itemize}
If $e=(r_1,c_1)$ is a 1-entry of $P$, then it is row-bounding.
\end{lemma}
\begin{figure}[!ht]
\centering
\includegraphics[width=0.9\textwidth]{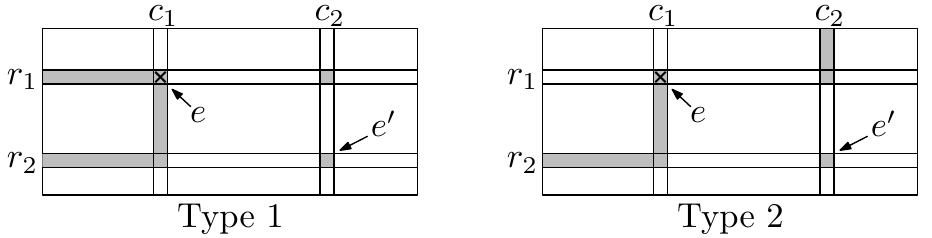}
\caption{Illustration of the proof of Lemma~\ref{lem-I2}. As before, the shaded 
areas correspond to possible locations of 1-entries, and the 1-entry $e$, 
marked by a cross, is row-bounding.}
\label{fig-I2}
\end{figure}
\begin{proof}
Suppose that $P\in\Pat$ satisfies the assumptions of the lemma, and that the 
entry $e=(r_1,c_1)$ is a 1-entry. Let $e'$ be the entry $(r_2,c_2)$ of~$P$.
Notice that if $e'$ is a 0-entry, we can deduce that $e$ is row-bounding by 
Lemma~\ref{lem-leftmost} (for Type~1) or by Lemma~\ref{lem-H} (for Type~2). 
Assume therefore that $e'$ is a 1-entry of~$P$.

Let $M\in\Mat$ be a $P$-avoider. We will show that every row of $M$ has at most 
$\ell(\ell+1)$ 0-runs critical for~$e$. Suppose that a row $r'$ of $M$ has more 
than $\ell(\ell+1)$ 0-runs critical for~$e$. Among these 0-runs, we select a 
subsequence $z_1,z_2,\dotsc,z_{\ell+1}$ numbered left to right, with the 
property that for each $i\in[\ell]$, $M$ has at least $\ell$ 1-entries in row 
$r'$ between $z_i$ and $z_{i+1}$, and $M$ also has at least $\ell$ 1-entries in 
row~$r'$ to the right of $z_{\ell+1}$.

For each $i\in[\ell+1]$, let $f_i$ be a 0-entry in $z_i$ critical for $e$, and 
let $\phi_i$ be an embedding of $P$ into $M\Delta f_i$ that maps $e$ to~$f_i$. 
For $i\in[\ell]$, let $w_i$ be the interval of entries that lie between $z_i$ 
and $z_{i+1}$ in row $r'$ of~$M$, and let $w_{\ell+1}$ be the interval of 
entries in row $r'$ to the right of~$z_{\ell+1}$. Recall that each $w_i$ 
contains at least $\ell$ 1-entries. Let $g_i$ be the leftmost entry in~$w_i$, 
which is necessarily a 1-entry, because $z_i$ is a maximal interval of 
0-entries. Finally, let $h_i=(p_i,q_i)$ be the 1-entry~$\phi_i(e')$ (recall that 
$e'=(r_2,c_2)$ is a 1-entry of~$P$). 

\begin{figure}
 \centerline{\includegraphics[width=\textwidth]{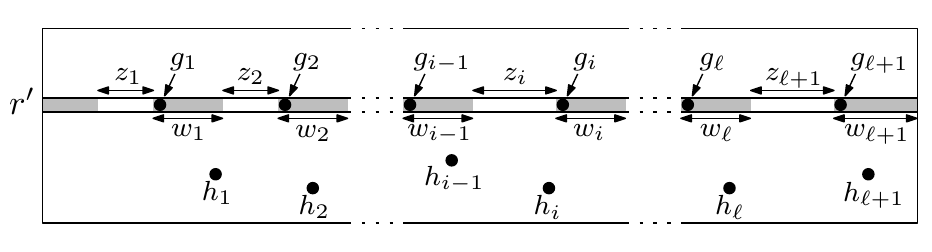}}
\caption{Illustration of the proof of Lemma~\ref{lem-I2}: the 
structure of a $P$-avoiding matrix with many intervals critical 
for $e$ in row~$r'$.}\label{figI2pf}
\end{figure}

Let us define a partial embedding $\psi$ of $P$ into~$M$. We let $\psi$ map the 
entry $(r_1,c_2)$ to the 1-entry $g_{\ell+1}$, and if $P$ is of Type 2, then for
every 1-entry $e''$ in the interval $[1,r_1)\times\{c_2\}$, we define 
$\psi(e'')=\phi_{\ell+1}(e'')$. Note that all the entries we mapped so far are 
to the right of~$f_{\ell+1}$.

To define $\psi$ for the remaining 1-entries of $P$, we will distinguish 
several situations, depending on the positions of the entries~$h_i=(p_i,q_i)$.

If, for some $i\in[\ell]$, the entry $h_i$ is to the right of the rightmost 
column of $w_i$, we put $\psi(e)=g_i$, and for every 1-entry $e''$ of $P$ for 
which $\psi$ has not yet been defined, we put $\psi(e'')=\phi_i(e'')$. To see 
that the mapping $\psi$ is a partial embedding of $P$ into $M$, it is enough to 
observe that all the 1-entries in column $c_2$ of $P$ are mapped by $\psi$ to 
entries strictly to the right of $w_{i}$, while by Lemma~\ref{lem-columns}, 
all the 1-entries in column $c_1$ are mapped to the columns intersecting the 
interval $z_i$, except for the entry $e$, which is mapped to~$g_i$. There are 
therefore at least $\ell-1$ columns which separate the image of any entry from 
column $c_1$ from the image of any entry from column~$c_2$. With this in mind, 
it is easy to check that $\psi$ is indeed a partial embedding.

Suppose that the situation from the previous paragraph does not occur, 
that is, for every $i\in[\ell]$, the entry $h_i$ is not to the right of 
the rightmost column intersecting~$w_i$. 
Since $h_i$ must by construction be to the right of the column containing 
$f_i$, we know that the column $q_i$ containing $h_i$ intersects either 
$z_i$ or~$w_i$. In particular, we have $q_1<q_2<\dotsb<q_{\ell+1}$. 

Assume now, that for some $i\in[\ell]$, the inequality $p_i\le p_{i+1}$ holds. 
We now complete the mapping $\psi$ as follows: we put $\psi(e)=g_i$,  
$\psi(e')=h_{i+1}$, and for all the 1-entries $e''$ of $P$ not yet mapped (i.e., 
the 1-entries in columns $1,\dotsc,c_1$ except $e$), we put 
$\psi(e'')=\phi_i(e'')$. The mapping $\phi$ is again a partial embedding of $P$ 
into~$M$.

It remains to deal with the situation when we have 
$p_1>p_2>\dotsb>p_\ell>p_{\ell+1}$, which means that the 1-entries 
$h_1,h_2,\dotsc,h_{\ell+1}$ form an image of the diagonal 
pattern~$\ovD_{\ell+1}$. We complete the mapping $\psi$ as follows: a 1-entry 
of the form $(r_1,j)$ for $j\le c_1$ is mapped to the entry $g_j$, a 1-entry of 
the form $(r_2,j)$ for any $j\in[\ell]$ is mapped to $h_j$, and any 1-entry 
$e''\in[r_1+1,r_2)\times\{c_1\}$ is mapped to $\phi_\ell(e'')$. Note that for 
$j<c_1$, the mapping $\psi$ maps the 1-entries in column $j$ to 1-entries in 
columns intersecting $z_j\cup w_j$, and for $j=c_1$, the 1-entries in column 
$j$ get mapped to columns intersecting $z_j\cup w_j\cup z_\ell$. 

In all cases, we found a partial embedding $\psi$ of $P$ into $M$, which is a 
contradiction. Therefore, each row of $M$ has at most $\ell(\ell+1)$ 0-runs 
critical for $e$, and $e$ is row-bounding.
\end{proof}

\paragraph{Row-boundedness of specific patterns.} We now have enough 
technical tools to establish that any pattern $P$ from $\Avm{\rbpat}$ is 
row-bounding. Recall from Proposition~\ref{pro:boundedints} that any 
$P\in\Avm{\rbpat}$ avoids $D_2$ or $\ovD_2$ or can be covered by three lines. 

We will first look at patterns that can be covered by fewer than three lines, 
and show that they are all row-bounding.

\begin{lemma}\label{lem-1row2col}
A pattern $P$ that has at most two nonempty columns or at most one nonempty row 
is row-bounding.
\end{lemma}
\begin{proof}
It follows from Lemma~\ref{lem-leftmost} and trivial symmetries that every 
1-entry of $P$ is row-bounding, hence $P$ is row-bounding.
\end{proof}

\begin{lemma}
\label{lem-tworows2}
If $P\in\Pat$ is a pattern with two nonempty rows, then $P$ is row-bounding.
\end{lemma}
\begin{proof}
We will show that for every 1-entry $e$ of $P$, we have $\rce{\Avm{P}}{e}\le 
\ell^2$.

In view of Observation~\ref{obs-empty}, we may assume that only the first row 
and the last row of $P$ are nonempty. Let $e$ be a 1-entry of $P$, and suppose 
without loss of generality that $e$ is in the first row, i.e., $e=(1,c)$ for 
some~$c$. 

Given a matrix~$M\in\Avm{P}$, consider an arbitrary row~$r$ of~$M$. For 
contradiction, suppose that the row $r$ has $\ell^2+1$ distinct 0-runs
$z_1,\dots,z_{\ell^2+1}$ critical for~$e$, numbered left to right. Let $c_i$ 
denote the leftmost 
column intersecting $z_i$, and for $i\le \ell^2$, let $X_i$ denote the set 
of column indices $[c_i, c_{i+1})$. Observe that for every $i\le\ell^2$, $M$ 
has at least one 1-entry in the interval $\{r\}\times X_i$.

\begin{figure}
 \centerline{\includegraphics[width=\textwidth]{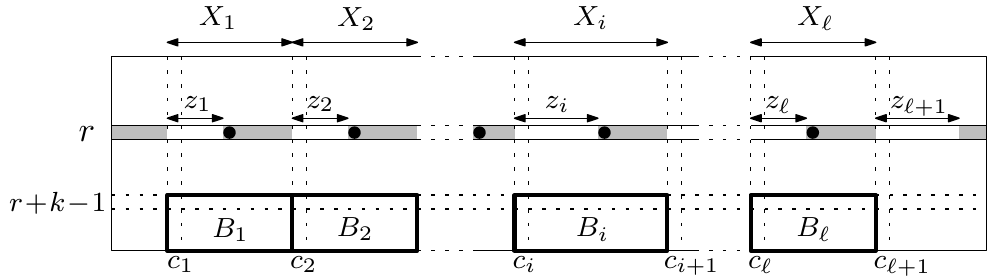}}
\caption{The matrix $M$ considered in the proof of 
Lemma~\ref{lem-tworows2}.}\label{fig-tworows}
\end{figure}

Let $B_i$ be the submatrix $M[[r+k-1,m]\times X_i]$ of~$M$ (see 
Figure~\ref{fig-tworows}). Note that if there are at least $\ell$ distinct 
values of $i$ for which $B_i$ contains at least one 1-entry, then the matrix $M$ 
contains the pattern~$P$. 

Suppose therefore that $B_i$ is empty for each $i$ up to at most $\ell-1$ 
exceptions. In particular, there is an index $j\in[\ell^2]$ such that the 
$\ell$ consecutive submatrices $B_j, B_{j+1}, \dotsc, B_{j+\ell-1}$ are all 
empty.

Recall that $e=(1,c)$ is a 1-entry of $P$, and that all the 1-entries of $P$ 
are in rows 1 and~$k$. Let $c'$ be a column index such that $e'=(k,c')$ is a 
1-entry of $P$, and $|c-c'|$ is as small as possible. Suppose without loss 
of generality that $c\le c'$ and let $d:=c'-c$. 

Let $f$ be a 0-entry in $z_j$ critical for $e$, and let $\phi$ be an embedding 
of $P$ into $M\Delta f$ that maps $e$ to~$f$. Note that by 
Lemma~\ref{lem-columns}, $\phi$ maps the entries in column $c$ of $P$ to 
entries in columns intersecting $z_j$, and in particular, the entry $(k,c)$ is 
mapped inside~$B_j$. Since $B_j$ is empty, $(k,c)$ is a 0-entry and 
in particular, $c'$ is greater than~$c$.

It follows that the 1-entry $e'=(k,c')$ is mapped strictly to the right of 
the column containing~$f$, and since $B_j,\dotsc,B_{j+\ell-1}$ are all empty, 
$e'$ must be mapped to the right of the columns in the set~$X_{j+\ell-1}$. 

We now define a partial embedding $\psi$ of $P$ into $M$ as follows: the $d+1$ 
entries in $P[\{1\}\times[c,c']]$ get mapped into $M[\{r\}\times (X_j\cup 
X_{j+1}\cup\dotsb\cup X_{j+d})]$ by~$\psi$ (recall that $\{r\}\times X_i$ 
contains at least one 1-entry for 
each $i$). The remaining 1-entries of $P$ are mapped by $\psi$ in the same way 
as by~$\phi$. Then $\psi$ is a partial embedding of $P$ into $M$, a 
contradiction.
\end{proof}

\begin{lemma}\label{lem-rowcol}
A pattern $P$ that can be covered by one row and one column is row-bounding.
\end{lemma}
\begin{proof}
Suppose that $P\in\Pat$ is covered by row $r$ and column $c$. By 
Lemma~\ref{lem-leftmost}, all the 1-entries in $P[\{r\}\times[c]]$ are 
row-bounding, and by symmetry, the 1-entries in $P[\{r\}\times[c,\ell]]$ are 
row-bounding as well. By Lemma~\ref{lem-I}, the 1-entries in 
$P[[1,r)\times\{c\}]$ and $P[(r,k]\times\{c\}]$ are also row-bounding.
\end{proof}

Lemmas~\ref{lem-1row2col}, \ref{lem-tworows2} and \ref{lem-rowcol} imply that 
any pattern that can be covered by two lines is row-bounding. We now proceed 
with the remaining cases of Proposition~\ref{pro:boundedints}.

\begin{lemma}
\label{lem-walkpat}
A pattern $P\in\Pat$ that avoids $D_2$ or $\ovD_2$ is row-bounding.
\end{lemma}
\begin{proof}
Suppose that $P$ avoids $\ovD_2$, the other case being symmetric.
From Proposition~\ref{pro-diag}, we know that $P$ is a decreasing pattern. 
Every 1-entry of $P$ is row-bounding either by Lemma~\ref{lem-H} (Type 3), or by 
 Lemma~\ref{lem-I} (Type 3), and therefore $P$ is row-bounding. 
\end{proof}

What follows is the last and the most difficult case of our analysis, which 
deals with patterns that are not increasing or decreasing and cannot be 
covered by two lines. 

\begin{lemma}
\label{lem-2types}
Let $P\in\Avm{\rbpat}$ be a pattern that contains both $D_2$ and $\ovD_2$, 
and that cannot be covered by two lines. Then $P$ can be transformed by a 
rotation or a reflection to a pattern $P_0$ of one of these two types (see 
Figure~\ref{fig-2types}).
\begin{itemize}
\item[Type 1:] $P_0$ has three rows $r<r'<r''$ and two columns $c<c'$ with 
\[
 \supp(P_0)\subseteq \big(\{r'\}\times [c,c'] \big)\cup\{(r,c), (r'',c), 
(r,c'),(r'',c')\}.
\]
\item[Type 2:] $P_0$ has two rows $r<r'$ and two columns $c<c'$ with
\[
\supp(P_0)\subseteq \big( \{r\}\times[c,c']\big)\cup \big( 
\{r'\}\times[c]\big)\cup \big([r]\times \{c'\}\big)\cup\{(r',c')\}.
\]
\end{itemize}
\end{lemma}

\begin{figure}
 \centerline{\includegraphics[width=0.9\textwidth]{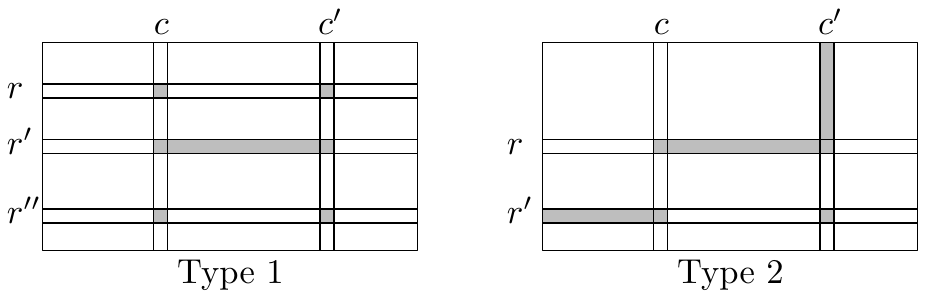}}
\caption{The two types of $\rbpat$-avoiders considered in 
Lemma~\ref{lem-2types}. The shaded areas are the possible 
positions of 1-entries.}\label{fig-2types}
\end{figure}

\begin{proof}
Let $P\in\Pat$ be a pattern satisfying the assumptions of the lemma. Since $P$ 
cannot be covered by two lines, by Fact~\ref{fac-eger}, $P$ contains three 
1-entries $e_1=(r_1,c_1)$, $e_2=(r_2,c_2)$ and $e_3=(r_3,c_3)$, with 
$r_1<r_2<r_3$, and such that the columns $c_1,c_2,c_3$ are all distinct. Since 
$P$ avoids the patterns from $\rbpat$, we must have either $c_1<c_2<c_3$ or 
$c_1>c_2>c_3$. Without loss of generality, assume $c_1<c_2<c_3$.

By Proposition~\ref{pro:boundedints}, $P$ can be covered by three lines. 
Suppose 
first that the three lines that cover $P$ are the rows $r_1$, $r_2$ and~$r_3$.
Suppose moreover, that the three 1-entries were chosen in such a way that $c_1$ 
is as large as possible, while $c_2$ and $c_3$ are as small as possible; see 
Figure~\ref{fig-2typespf} (left). In particular, row $r_1$ of $P$ has no 
1-entry in any of the columns $[c_1+1,c_2)$, otherwise we could choose a larger 
value of~$c_1$. Similarly, row $r_2$ has no 1-entry in columns $[c_1+1,c_2)$ 
and row $r_3$ has no 1-entry in columns $[c_2+1,c_3)$.

Moreover, since $P$ avoids the four patterns from the set $\rbpat$, row $r_1$ 
has no 1-entry in columns $[c_2+1,c_3)$ or $(c_3,\ell]$, row $r_2$ has no 
1-entry in columns $[1,c_1)$ or $(c_3,\ell]$, and row $r_3$ has no 1-entry 
in columns $[1,c_1)$ or~$[c_1+1,c_2)$.

Therefore, apart from the three 1-entries $e_i$, a 1-entry of $P$ can appear in 
one of the three intervals $\alpha=\{r_1\}\times[1,c_1)$,
$\beta=\{r_2\}\times(c_2,c_3]$ and $\gamma=\{r_3\}\times(c_3,\ell]$, or be 
equal to one of the five entries $a=(r_2,c_1)$, $b=(r_3,c_1)$, $c=(r_1,c_2)$, 
$d=(r_3,c_2)$ or $e=(r_1,c_3)$; see Figure~\ref{fig-2typespf} (left).
Note that $a$ and $c$ cannot be simultaneously equal to 1, otherwise they 
would form a forbidden pattern with $e_3$, and similarly, if $\beta$ contains a 
1-entry then $d=0$, if $\alpha$ contains a 1-entry then $b=0$, and if $\gamma$ 
contains a 1-entry then~$e=0$.

Since $P$ contains a copy of $\ovD_2$, at least one of $b$ and $e$ must be a 
1-entry. Let us go through the cases that may occur.

\paragraph{Case I: $b=1$.} If $b=1$ then $\alpha$ is empty. We have two 
subcases:
\begin{itemize}
\item[Ia: $\beta$ contains a 1-entry.] Then $c=0$ and $d=0$. If 
$\gamma$ is empty, then $P$ is a Type 1 matrix, with $c=c_1$, 
$c'=c_3$, and $(r,r',r'')=(r_1,r_2,r_3)$. If $\gamma$ is nonempty, then $e=0$, 
and $P$ is a mirror image of a Type 2 matrix, with $(r,r')=(r_2,r_3)$ and 
$(c,c')=(c_3,c_1)$.
\item[Ib: $\beta$ is empty.] If $\gamma$ is nonempty, then $e=0$ and since at 
most one of $a$ and $c$ is nonempty, rotating $P$ counterclockwise by 90 
degrees yields a Type 2 matrix. If $\gamma$ is empty, then either $a=0$ and $P$ 
is the transpose of a Type 1 matrix, or $a=1$, and therefore $c=0$, and at least 
one of $d$ and $e$ is a 0-entry, resulting in a Type 1 matrix or a rotated Type 2 
matrix.
\end{itemize}
\paragraph{Case II: $b=0$.} If $b=0$, then $e=1$, otherwise $P$ would avoid 
$\ovD_2$. Consequently, $\gamma$ is empty. If $\beta$ were 
empty as well, then $P$ would be symmetric to a matrix from case I by a 
180-degree rotation. We may therefore assume that $\beta$ is nonempty, and 
hence $d=0$. At most one of $a$ and $c$ can be a 1-entry, and in either case we 
get an upside-down copy of a Type 2 matrix.

\begin{figure}
 \centerline{\includegraphics[width=\textwidth]{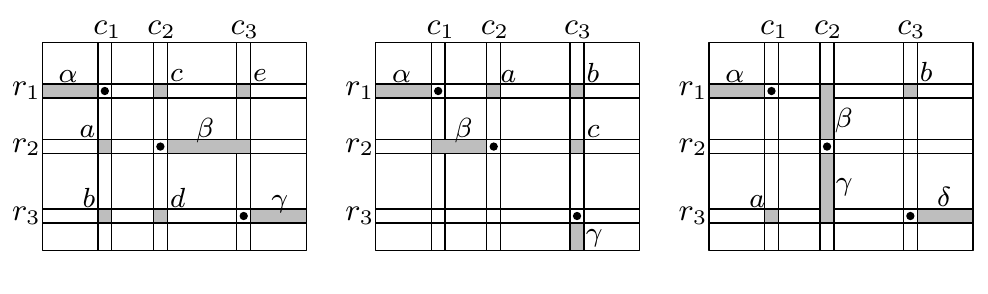}}
\caption{$\rbpat$-avoiders covered by rows $r_1$, $r_2$ and 
$r_3$ (left), by rows $r_1$, $r_2$ and column $c_3$ (center), and by rows 
$r_1$, $r_3$ and column $c_2$ (right). The shaded entries 
are potential 1-entries, the dots represent the three 1-entries $e_1$, $e_2$ 
and~$e_3$.}\label{fig-2typespf}
\end{figure}

This completes the analysis of matrices that can be covered by 3 rows. Suppose 
now that $P$ can be covered by two rows and one column. As each of the three 
entries $e_1$, $e_2$ and $e_3$ must be covered by a distinct line, there are 
three possibilities: either $P$ is covered by rows $r_1$ and $r_2$ and column 
$c_3$; or $P$ is covered by rows $r_1$ and $r_3$ and column $c_2$; or $P$ is 
covered by rows $r_2$ and $r_3$ and column~$c_1$. The last possibility is 
symmetric to the first one, so we only consider the first two.

Suppose $P$ is covered by rows $r_1$ and $r_2$ and column~$c_3$. Choose $c_1$ 
and $c_2$ to be as large as possible, and $r_3$ to be as small as possible. 
Together with the absence of patterns from~$\rbpat$, this means that apart from 
the 1-entries $e_1$, $e_2$ and $e_3$, all the remaining 1-entries must be inside 
the intervals $\alpha$, $\beta$ and $\gamma$ or at the positions $a$, $b$ or $c$ 
depicted in Figure~\ref{fig-2typespf} (center). Moreover, if $a=1$ then $\beta$ 
is empty. Therefore, $P$ is an upside-down copy of a matrix of Type 2, with the 
role of column $c$ played by $c_1$ if $a=0$ or by $c_2$ if~$a=1$.

Let us now suppose that $P$ is covered by rows $r_1$ and $r_3$ and column 
$c_2$. See Figure~\ref{fig-2typespf} (right). Suppose $c_1$ is largest possible 
and $c_3$ smallest possible. We make no assumptions about $r_2$, to keep the 
configuration symmetric. All the 1-entries are in the intervals $\alpha$, 
$\beta$, $\gamma$ and $\delta$ or at the positions $a$ and~$b$ depicted in the 
figure. Since $P$ contains $\ovD_2$, at least one of $a$ and $b$ is a 1-entry. 
Suppose without loss of generality that $a=1$. Then $\alpha$ is empty. If 
$\delta$ is nonempty, then $b=0$, and $P$ is a Type 2 matrix rotated 90 degrees 
clockwise. Otherwise $\delta$ is empty and $P$ is a rotated Type 1 matrix.

The cases when $P$ can be covered by three columns, or by two columns and a 
row, are symmetric to the cases handled so far by a 90-degree rotation.
\end{proof}

We now have all the ingredients to complete the proof of our main result.

\begin{figure}
 \centerline{\includegraphics{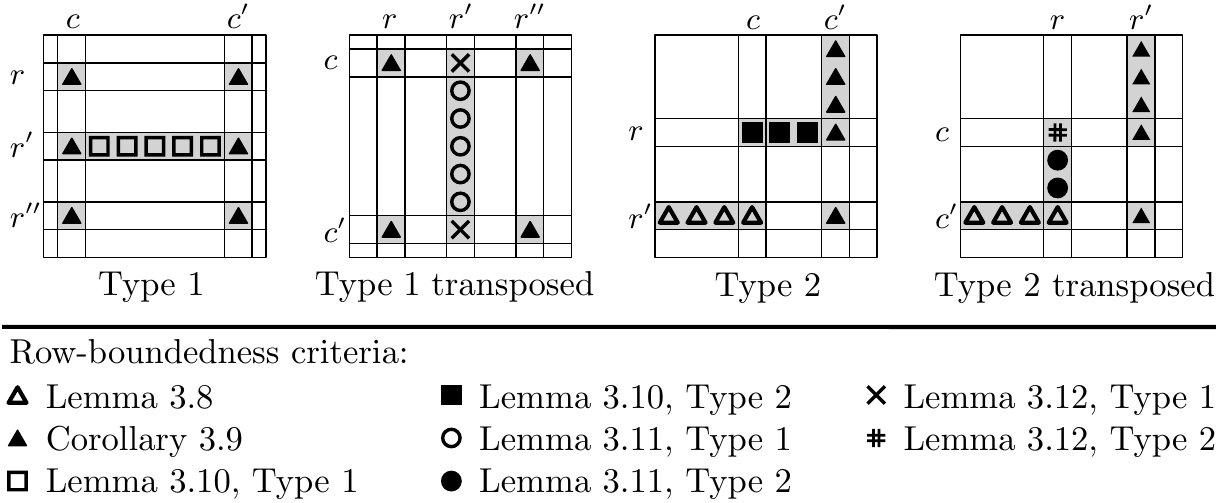}}
\caption{Illustration of the proof of Theorem~\ref{thm-rbound}. The 
symbols indicate the criteria used to prove row-boundedness of the 
1-entries in the two types of patterns of 
Lemma~\ref{lem-2types}.}\label{fig-rowbound}
\end{figure}

\begin{thm}\label{thm-rbound}
 Every pattern $P\in\Avm{\rbpat}$ is row-bounding. 
\end{thm}
\begin{proof}
Choose a $P\in\Avm{\rbpat}$. By Proposition~\ref{pro:boundedints}, either $P$ 
can be covered by three lines, or it avoids $D_2$, or it avoids 
$\ovD_2$. If $P$ avoids one of the two patterns of size 2, 
then it is row-bounding by Lemma~\ref{lem-walkpat}. If it can be covered by two 
lines, it is row-bounding by Lemmas~\ref{lem-1row2col}, \ref{lem-tworows2} 
and~\ref{lem-rowcol}.
Finally, if $P$ contains both $D_2$ and $\ovD_2$ and cannot be covered by two 
lines, Lemma~\ref{lem-2types} shows that, up to symmetry, $P$ corresponds to a 
matrix of Type 1 or Type 2. We therefore need to argue that the matrices of 
these two types, as well as their transposes, are row-bounding. See 
Figure~\ref{fig-rowbound}.

If $P$ is of Type 1, its 1-entries in column $c$ or in column $c'$ are 
row-bounding by Corollary~\ref{cor-leftmost}, and those in row $r'$ are 
row-bounding by Lemma~\ref{lem-H}. 

If $P$ is the transpose of a Type 1 matrix, then its 1-entries in columns $r$ 
and $r''$ are row-bounding by Corollary~\ref{cor-leftmost}, and those in column 
$r'$ by Lemmas \ref{lem-I} and~\ref{lem-I2}.

If $P$ is of Type 2, the 1-entries in row $r'$ and in column $c'$ are 
row-bounding by Lemma~\ref{lem-leftmost} and Corollary~\ref{cor-leftmost}, and 
those in row $r$ are row-bounding by Lemma~\ref{lem-H}.

Finally, if $P$ is the transpose of a Type 2 matrix, the 1-entries in 
column $r'$ and in row $c'$ are row-bounding by Lemma~\ref{lem-leftmost} 
and Corollary~\ref{cor-leftmost}, and the remaining 1-entries are covered by 
Lemmas \ref{lem-I} and~\ref{lem-I2}.
\end{proof}

Theorems~\ref{thm-unbound} and \ref{thm-rbound} together imply
Theorem~\ref{thm-main}.

\section{Further directions and open problems}\label{sec-further}

\paragraph{Boundedness of non-principal classes.}
So far, we only considered principal classes of matrices, i.e., classes 
determined by a single forbidden pattern. It is natural to ask to what extent 
our results generalize to arbitrary minor-closed classes of matrices, or at 
least to classes determined by a finite number of forbidden patterns. 

All our row-boundedness results for principal classes are based on the study of 
row-bounding 1-entries in a pattern~$P$. This approach extends straightforwardly 
to the setting of multiple forbidden patterns. In particular, for a set $\cF$ 
of patterns, a pattern $P\in\cF$ and a 1-entry $e$ of $P$, we say that $e$ is 
\emph{row-bounding in $\Avm{\cF}$} if each row of a matrix $M\in\Avm{\cF}$ 
has only a bounded number of 0-runs critical for~$e$ with respect to~$P$. Note 
that if $\cF$ is finite, then $\Avm{\cF}$ is row-bounded if and only if each 
1-entry of each pattern $P\in\cF$ is row-bounding in~$\Avm{\cF}$.

Note also that, by definition, if $e$ is a 1-entry of $P$ that is row-bounding 
in $\Avm{P}$, then for every set of patterns $\cF$ that contains $P$, the entry 
$e$ is also row-bounding in $\Avm{\cF}$, since $\Avm{\cF}$ is a subclass of 
$\Avm{P}$. Therefore, all the criteria for row-bounding entries that we derived 
in Subsection~\ref{ssec-bound} are applicable to non-principal classes as well.

We have seen in Corollary~\ref{cor-first} that a principal class is row-bounded 
if and only if it is column-bounded. Our next example shows that this property 
does not generalize to non-principal classes.

\begin{prop}\label{pro-counter} For the set of patterns $\cF=\{D_4,P\}$ with
\[
 P=\smm{ &\bullet& \\\bullet& & \\ &\bullet& \\ & 
&\bullet} \text{ and } D_4=\smm{\bullet& & & \\ &\bullet& & \\ & &\bullet& \\ 
& &&\bullet},
\]
the class $\Avm{\cF}$ is row-bounded but not column-bounded.
\end{prop}
\begin{proof}
To prove that $\Avm{\cF}$ is not column-bounded, we apply the transpose of the 
construction of Theorem~\ref{thm-unbound}, and observe that the constructed 
matrix avoids~$D_4$ (see Figure~\ref{fig-colunbound} (left)).

To prove that the class~$\Avm{\cF}$ is row-bounded, observe first that all the 
1-entries in $D_4$ are row-bounding by Lemma~\ref{lem-H}, the leftmost and the 
rightmost 1-entry of $P$ are row-bounding by Corollary~\ref{cor-leftmost}, and 
the 1-entry $(3,2)$ of $P$ is row-bounding by Lemma~\ref{lem-I}. It thus 
remains to show that the entry $e=(1,2)$ of $P$ is row-bounding in $\Avm{\cF}$.

\begin{figure}
 \centerline{\includegraphics[height=0.25\textwidth]{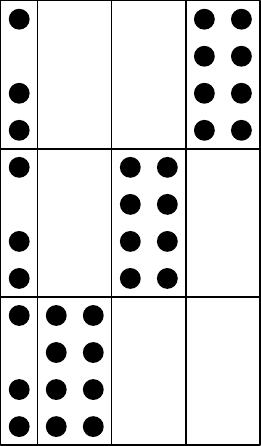} \hfil
\includegraphics[height=0.25\textwidth]{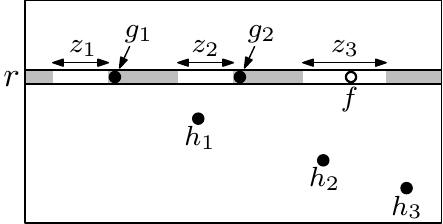}}
 \caption{Left: illustration that $\Avm{\cF}$ has unbounded column-complexity 
relative to the entry $e=(2,1)$ of $P$. Right: 
illustration of the proof that $\Avm{\cF}$ is 
row-bounded.}\label{fig-colunbound}
\end{figure}

We will show that each matrix $M\in\Avm{\cF}$ has at most two 0-runs critical 
for $e$ in any given row. Refer to Figure~\ref{fig-colunbound} (right). For 
contradiction, suppose that there are three 0-runs $z_1< z_2< z_3$ in a row $r$ 
of $M$. Let $g_1$ be a 1-entry of $M$ that lies in row $r$ between $z_1$ and 
$z_2$, let $g_2$ be a 1-entry of $M$ in row $r$ between $z_2$ and $z_3$, and let 
$f$ be a 0-entry in $z_3$ critical for the entry~$e$. Let $\phi$ be an embedding 
of $P$ into $M\Delta f$ with $\phi(e)=f$. 

Consider the three 1-entries $h_1=\phi(2,1)$, $h_2=\phi(3,2)$ and 
$h_3=\phi(4,3)$. If $h_1$ is in a column strictly to the right of $g_1$, then 
$g_1$ forms an image of $D_4$ with the three $h_i$s, a contradiction. If, on 
the 
other hand, $h_1$ is not to the right of $g_1$, then $h_1$ is strictly to the 
left of $g_2$, and $g_2$ forms an image of $P$ with the three $h_i$s (recall 
that $h_3$ is to the right of $f=\phi(1,2)$, and therefore also to the right 
of~$g_2$). This shows that $\Avm{\cF}$ is row-bounded.
\end{proof}

Recall from Corollary~\ref{cor-second}, that any principal subclass of a 
bounded principal class is again bounded. The example of 
Proposition~\ref{pro-counter} shows that this result does not generalize to 
non-principal classes: indeed, the class $\Avm{D_4}$ is bounded by 
Theorem~\ref{thm-main} (or by Corollary~\ref{cor-diag}), while its subclass 
$\Avm{\cF}$ is not bounded.

On the positive side, it is not hard to show that row-boundedness (and 
therefore also boundedness) is closed under union and intersection of classes. 

\begin{prop}
\label{pro-boundunion}
If $\cC_1$ and $\cC_2$ are row-bounded classes of matrices, then the classes 
$\cC_1\cup\cC_2$ and $\cC_1\cap\cC_2$ are row-bounded as well.
\end{prop}
\begin{proof}
Let $K_i$ be the row-complexity of the class $\cC_i$, for $i\in\{1,2\}$. Since 
every matrix that is critical for $\cC_1\cup\cC_2$ is also critical for $\cC_1$ 
or for $\cC_2$, we observe that $\cC_1\cup\cC_2$ has row-complexity at most 
$\max\{K_1,K_2\}$. In particular, $\cC_1\cup\cC_2$ is row-bounded.

Let us argue that $\cC_1\cap\cC_2$ is row-bounded as well. We claim that 
$\cC_1\cap\cC_2$ has row-complexity at most $K:=K_1+K_2$. For contradiction, 
suppose that there is a matrix $M$ critical for $\cC_1\cap\cC_2$ with 
row-complexity at least $K+1$. 

Let $r$ be a row of $M$ with maximum complexity, let $z_1, z_2,\dotsc,z_{K+1}$ 
be a sequence of 0-runs in this row, and let $f_i$ be a 0-entry in~$z_i$. By 
criticality of~$M$, we know that for each $i\in[K+1]$, the matrix $M\Delta f_i$ 
does not belong to $\cC_1$ or does not belong to~$\cC_2$. 

In particular, there are either at least $K_1+1$ values of $i$ for which 
$M\Delta f_i$ is not in $\cC_1$, or at least $K_2+1$ values of $i$ for which 
$M\Delta f_i$ is not in $\cC_2$. Suppose without loss of generality that the 
former situation occurs. Let $M^+$ be a critical matrix for the class $\cC_1$ 
that dominates the matrix~$M$. If $f_i$ is a 0-entry of $M$ such $M\Delta f_i$ 
is not in $\cC_1$, then $f_i$ is also a 0-entry of $M^+$. It follows that $M^+$ 
has at least $K_1+1$ 0-runs in row $r$, which is impossible, since $K_1$ is the 
row-complexity of~$\cC_1$.
\end{proof}

In contrast with Proposition~\ref{pro-boundunion}, an intersection of two 
unbounded classes is not necessarily unbounded, as we will now show. Consider 
the two patterns $Q_1=\smm{ &\bullet& 
\\\bullet& & \\ & &\bullet}$ and $Q_2=\smm{ &\bullet& \\ & 
&\bullet\\\bullet& & }$, and recall from Theorem~\ref{thm-main} that both 
$\Avm{Q_1}$ and $\Avm{Q_2}$ are unbounded classes.

\begin{prop}\label{pro-unbinter}
The class $\Avm{\{Q_1,Q_2\}}=\Avm{Q_1}\cap\Avm{Q_2}$ is bounded.
\end{prop}
\begin{proof}
Let us first show that every 1-entry of the two patterns $Q_1$ and $Q_2$ is 
row-bounding for $\cC:=\Avm{\{Q_1,Q_2\}}$. For a 1-entry that belongs to the 
first or the last column of either pattern, this follows from 
Corollary~\ref{cor-leftmost}. 

Consider the 1-entry $e=(1,2)$ of the pattern $Q_1$. We claim that each 
row in a matrix $M\in\cC$ has at most two 0-runs critical for~$e$. Suppose 
that a matrix $M\in\cC$ has a row $r$ with three 0-runs 
$z_1<z_2<z_3$ critical for~$e$. Let $f_i$ be a 0-entry in $z_i$ critical for 
$e$, and let $g_i$ be a 1-entry in row $r$ between $z_i$ and $z_{i+1}$, for 
$i\in\{1,2\}$. 

For $i\in\{1,2,3\}$, let $\phi_i$ be an embedding of $Q_1$ into $M\Delta f_i$ 
that maps $e$ to~$f_i$. Consider the three 1-entries $h_1=\phi_1(2,1)$, 
$h_2=\phi_1(3,3)$, and $h_3=\phi_3(3,3)$. Let $p_i$ and $q_i$ be the row and 
the column containing~$h_i$.

\begin{figure}
 \centerline{\includegraphics[width=0.9\textwidth]{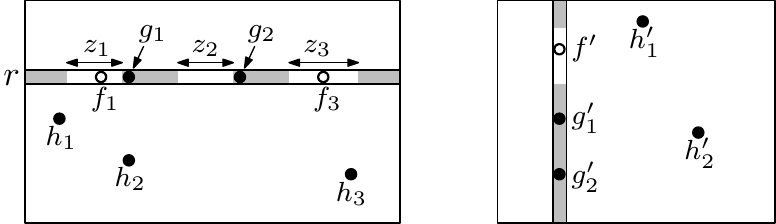}}
 \caption{Illustration of the row-boundedness (left) and column-boundedness 
(right) of $\Avm{\{Q_1,Q_2\}}$.}\label{fig-unbinter}
\end{figure}

Note that $g_1$ cannot be to the left of column $q_2$, since then $h_1$, $g_1$ 
and $h_2$ would form an image of~$Q_1$. It follows that $g_1$ is in the column 
$q_2$ or to the right of it, and consequently, we have $q_2<q_3$. Moreover, if 
$p_3>p_1$, then $h_1$, $g_2$ and $h_3$ form an image of $Q_1$, so $p_3$ is no 
larger than $p_1$ and hence $p_3<p_2$. But then $h_2$, $g_2$ and $h_3$ form an 
image of $Q_2$, a contradiction. 

By symmetry, the 1-entry $(1,2)$ of $Q_2$ is row-bounding as well, and 
therefore $\cC$ is row-bounded.

Let us now argue that $\cC$ is column-bounded. It is enough to show that the 
1-entry $e'=(2,1)$ of $Q_1$ is column-bounding for $\cC$, the rest follows from 
symmetry and from Corollary~\ref{cor-leftmost}. Suppose that a matrix $M\in\cC$ 
has a column $c$ with three 0-runs critical for~$e'$. In particular, column $c$ 
contains a 0-entry $f'$ critical for $e'$ such that below~$f'$, there are at 
least two 1-entries $g'_1$ and $g'_2$ in column $c$ of~$M$. Suppose that $g'_1$ 
is above~$g'_2$. 

Let $\phi$ be an embedding of $Q_1$ into $M\Delta f'$ with $\phi(e')=f'$. 
Define 
$h'_1=\phi(1,2)$ and $h'_2=\phi(3,3)$. Let $r'$ be the row containing~$h'_2$. 
If $g'_1$ is above row $r'$, then $g'_1$, $h'_1$ and $h'_2$ form a copy 
of~$Q_1$, and if $g'_1$ is not above row $r'$, then $g'_2$ is below row $r'$ 
and $g'_2$, $h'_1$ and $h'_2$ form a copy of $Q_2$, a contradiction.
\end{proof}

\paragraph{Open problems.}
A natural question arising from our results is to extend the dichotomy of 
Theorem~\ref{thm-main} to non-principal classes of matrices.
\begin{problem}
 For which sets $\cF$ of patterns is the class $\Avm{\cF}$ row-bounded? Can we 
characterize such sets $\cF$, at least when $\cF$ is finite?
\end{problem}

The notion of complexity we used in this paper is quite crude, in the sense that 
it only takes into account single lines of the corresponding matrix. It is 
reasonable to expect that matrices from a class of unbounded complexity possess 
nontrivial properties that could be revealed by a more refined approach.
\begin{problem}
Is there a refinement of our complexity notion that would provide nontrivial 
insight into the structure of critical matrices in unbounded classes?
\end{problem}

Throughout the paper, we focused on distinguishing bounded classes from 
unbounded ones. We made no attempts to obtain tight estimates for the actual 
value of the complexity of a bounded class. This might be a line of research 
worth pursuing.
\begin{problem}
What is the highest possible value of $\rc{\Avm{P}}$, over all row-bounding 
patterns $P$ of a given size $k\times \ell$? For which pattern is this maximum 
attained?
\end{problem}

By Observation~\ref{obs-empty},  if $P^+$ is a pattern
obtained by adding an empty row or column to the boundary of a pattern $P$, 
then $\Avm{P}$ has the same complexity as~$\Avm{P^+}$, and the avoiders of $P^+$ 
can be easily described in terms of the avoiders of~$P$. 

It is, however, more challenging to deal with a pattern 
$P^+$ obtained by inserting an empty line into the interior of a 
pattern~$P$. Theorem~\ref{thm-main} implies that $P$ is bounding if and only if 
$P^+$ is bounding, but we are not aware of any direct proof of this.
\begin{problem}
Let $P^+$ be a pattern obtained from a pattern $P$ by inserting a new empty row 
or column to an arbitrary position inside~$P$. Can we bound $\rc{\Avm{P^+}}$ in 
terms of $\rc{\Avm{P}}$? Can we describe the avoiders of $P^+$ in terms of the 
avoiders of $P$? If $\cF$ is a set of patterns and $\cF^+$ a set of patterns 
obtained by inserting empty rows and columns to the patterns in $\cF$, is it 
true that $\Avm{\cF^+}$ is bounded if and only if $\Avm{\cF}$ is?
\end{problem}

\footnotesize
\bibliographystyle{plain}     


\bibliography{matrices}
\end{document}